\documentclass{birkjour}
\usepackage{amsmath}
\usepackage{amsxtra}
\usepackage{amscd}
\usepackage{amsthm}
\usepackage{amsfonts}
\usepackage{amssymb}
\usepackage{eucal}

\newtheorem{cor}[subsection]{Corollary}

\newtheorem{lem}[subsection]{Lemma}
\newtheorem{prop}[subsection]{Proposition}

\newtheorem{conj}[subsection]{Conjecture}
\newtheorem{thm}[subsection]{Theorem}

\newtheorem{rem}[subsection]{Remark}

\theoremstyle{definition}

\theoremstyle{remark}

\newcommand{\nc}{\newcommand}
\nc{\renc}{\renewcommand} \nc{\ssec}{\subsection}
\nc{\sssec}{\subsubsection} \nc{\on}{\operatorname}

\nc\ol{\overline} \nc\ul{\underline} \nc\wt{\widetilde}
\nc\tboxtimes{\wt{\boxtimes}} \nc{\alp}{\alpha}

\nc{\ZZ}{{\mathbb Z}} \nc{\NN}{{\mathbb N}} \nc{\CC}{{\mathbb C}}
\nc{\OO}{{\mathbb O}} \renc{\SS}{{\mathbb S}} \nc{\DD}{{\mathbb
D}}
\renewcommand{\AA}{{\mathbb A}}
\nc{\Fq}{{\mathbb F}_q} \nc{\Fqb}{\ol{{\mathbb F}_q}}
\nc{\Ql}{\ol{{\mathbb Q}_\ell}} \nc{\id}{\text{id}} \nc\X{\mathcal
X}

\nc{\Hom}{\on{Hom}} \nc{\Lie}{\on{Lie}} \nc{\Loc}{\on{Loc}}
\nc{\Pic}{\on{Pic}} \nc{\Bun}{\on{Bun}} \nc{\IC}{\on{IC}}
\nc{\Aut}{\on{Aut}} \nc{\rk}{\on{rk}} \nc{\Sh}{\on{Sh}}
\nc{\Perv}{\on{Perv}} \nc{\pos}{{\on{pos}}} \nc{\Conv}{\on{Conv}}
\nc{\Sph}{\on{Sph}} \nc{\Sym}{\on{Sym}}
\nc{\BunBb}{\overline{\Bun}_B} \nc{\Buno}{\overset{o}{\Bun}}
\nc{\BunPb}{{\overline{\Bun}_P}}
\nc{\BunBM}{\overline{\Bun}_{B(M)}}
\nc{\BunPbw}{{\widetilde{\Bun}_P}}
\nc{\BunBP}{\widetilde{\Bun}_{B,P}} \nc{\GUb}{\overline{G/U}}
\nc{\GUPb}{\overline{G/U(P)}}
\nc{\iso}{{\stackrel{\sim}{\longrightarrow}}}

\nc{\Hhom}{\underline{\on{Hom}}} \nc\syminfty{\on{Sym}^{\infty}}
\nc\lal{\ol{\lambda}} \nc\xl{\ol{x}} \nc\thl{\ol{\theta}}
\nc\nul{\ol{\nu}} \nc\mul{\ol{\mu}} \nc\Sum\Sigma
\nc{\oX}{\overset{o}{X}{}}

\nc{\M}{{\mathcal M}} \nc{\N}{{\mathcal N}} \nc{\F}{{\mathcal F}}
\nc{\D}{{\mathcal D}} \nc{\Q}{{\mathcal Q}} \nc{\Y}{{\mathcal Y}}
\nc{\G}{{\mathcal G}} \nc{\E}{{\mathcal E}} \nc{\CalC}{{\mathcal
C}}
\nc\Dh{\widehat{\D}}
\renewcommand{\O}{{\mathcal O}}
\nc{\C}{{\mathcal C}} \nc{\K}{{\mathcal K}}
\renewcommand{\H}{{\mathcal H}}

\nc{\T}{{\mathcal T}} \nc{\V}{{\mathcal V}} \renc{\P}{{\mathcal
P}} \nc{\A}{{\mathcal A}} \nc{\B}{{\mathcal B}} \nc{\U}{{\mathcal
U}}
\renewcommand{\L}{{\mathcal L}}
\nc{\Gr}{\on{Gr}}

\nc{\frn}{{\check{\mathfrak u}(P)}}
\nc\f{{\mathfrak f}}

\nc{\q}{{\mathfrak q}} \nc{\p}{{\mathfrak p}} \nc{\s}{{\mathfrak
s}} \nc\w{\text{w}}

\nc\mathi\iota \nc\Spec{\on{Spec}} \nc\Mod{\on{Mod}}
\nc{\tw}{\widetilde{\mathfrak t}} \nc{\pw}{\widetilde{\mathfrak
p}} \nc{\qw}{\widetilde{\mathfrak q}} \nc{\jw}{\widetilde j}

\nc{\grb}{\overline{\Gr}} \nc{\I}{\mathcal I}

\nc{\lambdach}{{\check\lambda}} \nc{\Lambdach}{{\check\Lambda}{}}
\nc{\much}{{\check\mu}} \nc{\omegach}{{\check\omega}}
\nc{\nuch}{{\check\nu}} \nc{\etach}{{\check\eta}}
\nc{\alphach}{{\check\alpha}} \nc{\betach}{{\check\beta}}
\nc{\rhoch}{{\check\rho}} \nc{\ch}{{\check h}}

\nc{\Hb}{\overline{\H}}

\nc{\BA}{{\mathbb{A}}} \nc{\BC}{{\mathbb{C}}}
\nc{\BM}{{\mathbb{M}}} \nc{\BN}{{\mathbb{N}}}
\nc{\BP}{{\mathbb{P}}} \nc{\BR}{{\mathbb{R}}}
\nc{\BZ}{{\mathbb{Z}}} \nc{\BS}{{\mathbb{S}}}

\nc{\CA}{{\mathcal{A}}} \nc{\CB}{{\mathcal{B}}}
\nc{\CE}{{\mathcal{E}}} \nc{\CF}{{\mathcal{F}}}
\nc{\CG}{{\mathcal{G}}} \nc{\CH}{{\mathcal{H}}}
\nc{\CI}{{\mathcal{I}}} \nc{\CL}{{\mathcal{L}}}
\nc{\CM}{{\mathcal{M}}} \nc{\CN}{{\mathcal{N}}}
\nc{\CO}{{\mathcal{O}}} \nc{\CP}{{\mathcal{P}}}
\nc{\CQ}{{\mathcal{Q}}} \nc{\CR}{{\mathcal{R}}}
\nc{\CS}{{\mathcal{S}}} \nc{\CT}{{\mathcal{T}}}
\nc{\CU}{{\mathcal{U}}} \nc{\CV}{{\mathcal{V}}}
\nc{\CW}{{\mathcal{W}}} \nc{\CZ}{{\mathcal{Z}}}

\nc{\cM}{{\check{\mathcal M}}{}} \nc{\csM}{{\check{\mathcal A}}{}}
\nc{\oM}{{\overset{\circ}{\mathcal M}}{}}
\nc{\obM}{{\overset{\circ}{\mathbf M}}{}}
\nc{\oCA}{{\overset{\circ}{\mathcal A}}{}}
\nc{\obA}{{\overset{\circ}{\mathbf A}}{}}
\nc{\ooM}{{\overset{\circ}{M}}{}}
\nc{\osM}{{\overset{\circ}{\mathsf M}}{}}
\nc{\vM}{{\overset{\bullet}{\mathcal M}}{}}
\nc{\nM}{{\underset{\bullet}{\mathcal M}}{}}
\nc{\oD}{{\overset{\circ}{\mathcal D}}{}}
\nc{\obD}{{\overset{\circ}{\mathbf D}}{}}
\nc{\oA}{{\overset{\circ}{\mathbb A}}{}}
\nc{\op}{{\overset{\bullet}{\mathbf p}}{}}
\nc{\cp}{{\overset{\circ}{\mathbf p}}{}}
\nc{\oU}{{\overset{\bullet}{\mathcal U}}{}}
\nc{\oZ}{{\overset{\circ}{\mathcal Z}}{}}
\nc{\ofZ}{{\overset{\circ}{\mathfrak Z}}{}}

\nc{\ff}{{\mathfrak{f}}} \nc{\fv}{{\mathfrak{v}}}
\nc{\fa}{{\mathfrak{a}}} \nc{\fb}{{\mathfrak{b}}}
\nc{\fd}{{\mathfrak{d}}} \nc{\fe}{{\mathfrak{e}}}
\nc{\fg}{{\mathfrak{g}}} \nc{\fgl}{{\mathfrak{gl}}}
\nc{\fh}{{\mathfrak{h}}} \nc{\fri}{{\mathfrak{i}}}
\nc{\fj}{{\mathfrak{j}}} \nc{\fk}{{\mathfrak{k}}}
\nc{\fm}{{\mathfrak{m}}} \nc{\fn}{{\mathfrak{n}}}
\nc{\ft}{{\mathfrak{t}}} \nc{\fu}{{\mathfrak{u}}}
\nc{\fw}{{\mathfrak{w}}} \nc{\fz}{{\mathfrak{z}}}
\nc{\fp}{{\mathfrak{p}}} \nc{\frr}{{\mathfrak{r}}}
\nc{\fs}{{\mathfrak{s}}} \nc{\fsl}{{\mathfrak{sl}}}
\nc{\hsl}{{\widehat{\mathfrak{sl}}}}
\nc{\hgl}{{\widehat{\mathfrak{gl}}}}
\nc{\hg}{{\widehat{\mathfrak{g}}}}
\nc{\chg}{{\widehat{\mathfrak{g}}}{}^\vee}
\nc{\hn}{{\widehat{\mathfrak{n}}}}
\nc{\chn}{{\widehat{\mathfrak{n}}}{}^\vee}

\nc{\fA}{{\mathfrak{A}}} \nc{\fB}{{\mathfrak{B}}}
\nc{\fD}{{\mathfrak{D}}} \nc{\fE}{{\mathfrak{E}}}
\nc{\fF}{{\mathfrak{F}}} \nc{\fG}{{\mathfrak{G}}} \nc{\fH}{{\mathfrak{H}}}
\nc{\fI}{{\mathfrak{I}}} \nc{\fJ}{{\mathfrak{J}}}
\nc{\fK}{{\mathfrak{K}}} \nc{\fL}{{\mathfrak{L}}}
\nc{\fM}{{\mathfrak{M}}} \nc{\fN}{{\mathfrak{N}}}
\nc{\frP}{{\mathfrak{P}}} \nc{\fQ}{{\mathfrak{Q}}}
\nc{\fT}{{\mathfrak{T}}} \nc{\fU}{{\mathfrak{U}}}
\nc{\fV}{{\mathfrak{V}}} \nc{\fW}{{\mathfrak{W}}}
\nc{\fX}{{\mathfrak{X}}} \nc{\fY}{{\mathfrak{Y}}}
\nc{\fZ}{{\mathfrak{Z}}}

\nc{\ba}{{\mathbf{a}}}
\nc{\bb}{{\mathbf{b}}} \nc{\bc}{{\mathbf{c}}}
\nc{\be}{{\mathbf{e}}} \nc{\bj}{{\mathbf{j}}}
\nc{\bn}{{\mathbf{n}}} \nc{\bp}{{\mathbf{p}}}
\nc{\bq}{{\mathbf{q}}} \nc{\bfr}{{\mathbf{r}}}
\nc{\bfu}{{\mathbf{u}}} \nc{\bv}{{\mathbf{v}}}
\nc{\bx}{{\mathbf{x}}} \nc{\by}{{\mathbf{y}}}
\nc{\bw}{{\mathbf{w}}} \nc{\bA}{{\mathbf{A}}}
\nc{\bB}{{\mathbf{B}}} \nc{\bC}{{\mathbf{C}}}
\nc{\bD}{{\mathbf{D}}} \nc{\bF}{{\mathbf{F}}}
\nc{\bH}{{\mathbf{H}}} \nc{\bK}{{\mathbf{K}}}
\nc{\bM}{{\mathbf{M}}} \nc{\bN}{{\mathbf{N}}}
\nc{\bO}{{\mathbf{O}}} \nc{\bS}{{\mathbf{S}}}
\nc{\bV}{{\mathbf{V}}} \nc{\bW}{{\mathbf{W}}}
\nc{\bX}{{\mathbf{X}}}
\nc{\bY}{{\mathbf{Y}}} \nc{\bP}{{\mathbf{P}}}
\nc{\bZ}{{\mathbf{Z}}} \nc{\bh}{{\mathbf{h}}}

\nc{\sA}{{\mathsf{A}}} \nc{\sB}{{\mathsf{B}}}
\nc{\sC}{{\mathsf{C}}} \nc{\sD}{{\mathsf{D}}}
\nc{\sE}{{\mathsf{E}}} \nc{\sF}{{\mathsf{F}}}
\nc{\sK}{{\mathsf{K}}} \nc{\sL}{{\mathsf{L}}}
\nc{\sM}{{\mathsf{M}}} \nc{\sO}{{\mathsf{O}}}
\nc{\sQ}{{\mathsf{Q}}} \nc{\sP}{{\mathsf{P}}}
\nc{\sT}{{\mathsf{T}}} \nc{\sZ}{{\mathsf{Z}}}
\nc{\sV}{{\mathsf{V}}}
\nc{\sfp}{{\mathsf{p}}} \nc{\sr}{{\mathsf{r}}}
\nc{\st}{{\mathsf{t}}} \nc{\sfb}{{\mathsf{b}}}
\nc{\sfc}{{\mathsf{c}}} \nc{\sd}{{\mathsf{d}}}
\nc{\sz}{{\mathsf{z}}}

\nc{\BK}{{\bar{K}}}

\nc{\tA}{{\widetilde{\mathbf{A}}}}
\nc{\tB}{{\widetilde{\mathcal{B}}}}
\nc{\tg}{{\widetilde{\mathfrak{g}}}} \nc{\tG}{{\widetilde{G}}}
\nc{\TM}{{\widetilde{\mathbb{M}}}{}}
\nc{\tO}{{\widetilde{\mathsf{O}}}{}}
\nc{\tU}{{\widetilde{\mathfrak{U}}}{}} \nc{\TZ}{{\tilde{Z}}}
\nc{\tx}{{\tilde{x}}} \nc{\tbv}{{\tilde{\bv}}}
\nc{\tfP}{{\widetilde{\mathfrak{P}}}{}} \nc{\tz}{{\tilde{\zeta}}}
\nc{\tmu}{{\tilde{\mu}}}

\nc{\urho}{\underline{\rho}} \nc{\uB}{\underline{B}}
\nc{\uC}{{\underline{\mathbb{C}}}} \nc{\ui}{\underline{i}}
\nc{\uj}{\underline{j}} \nc{\ofP}{{\overline{\mathfrak{P}}}}
\nc{\oB}{{\overline{\mathcal{B}}}}
\nc{\og}{{\overline{\mathfrak{g}}}} \nc{\oI}{{\overline{I}}}

\nc{\eps}{\varepsilon} \nc{\hrho}{{\hat{\rho}}}
\nc{\blambda}{{\boldsymbol{\lambda}}}

\nc{\one}{{\mathbf{1}}} \nc{\two}{{\mathbf{t}}}

\nc{\Rep}{{\mathop{\operatorname{\rm Rep}}}}
\nc{\Tot}{{\mathop{\operatorname{\rm Tot}}}}
\nc{\Ker}{{\mathop{\operatorname{\rm Ker}}}}
\nc{\Hilb}{{\mathop{\operatorname{\rm Hilb}}}}
\nc{\End}{{\mathop{\operatorname{\rm End}}}}
\nc{\Ext}{{\mathop{\operatorname{\rm Ext}}}}
\nc{\CHom}{{\mathop{\operatorname{{\mathcal{H}}\it om}}}}
\nc{\GL}{{\mathop{\operatorname{\rm GL}}}}
\nc{\gr}{{\mathop{\operatorname{\rm gr}}}}
\nc{\Id}{{\mathop{\operatorname{\rm Id}}}}
\nc{\defi}{{\mathop{\operatorname{\rm def}}}}
\nc{\length}{{\mathop{\operatorname{\rm length}}}}
\nc{\supp}{{\mathop{\operatorname{\rm supp}}}}

\nc{\Cliff}{{\mathsf{Cliff}}}
\nc{\Fl}{{\mathsf{Fl}}} \nc{\Fib}{{\mathsf{Fib}}}
\nc{\Coh}{{\mathsf{Coh}}} \nc{\FCoh}{{\mathsf{FCoh}}}

\nc{\reg}{{\text{\rm reg}}}

\nc{\cplus}{{\mathbf{C}_+}} \nc{\cminus}{{\mathbf{C}_-}}
\nc{\cthree}{{\mathbf{C}_*}} \nc{\Qbar}{{\bar{Q}}}

\nc{\bOmega}{{\overline{\Omega}}}

\nc{\seq}[1]{\stackrel{#1}{\sim}}

\nc{\aff}{\operatorname{aff}}

\def\dsp{\displaystyle}

\begin{document}

\author{Boris Feigin, Michael Finkelberg, Andrei Negut and Leonid Rybnikov}
\title
{Yangians and cohomology rings of Laumon spaces}

\dedicatory{To our friend Sasha Shen on his 50th birthday}




\address{{\it Address}:\newline
B.F.: Landau Institute for Theoretical Physics, Kosygina st 2,
Moscow 117940, Russia \newline
M.F.: IMU, IITP, and
State University Higher School of Economics, \newline
Department of Mathematics, \newline
20 Myasnitskaya st,
Moscow 101000, Russia \newline
A.N.: Simion Stoilow Institute of Mathematics of the Romanian Academy,
Calea Grivitei nr. 21, Bucuresti, 010702, Romania \newline
L.R.: Institute for the Information Transmission Problems and \newline
State University Higher School of Economics, \newline
Department of Mathematics, \newline
20 Myasnitskaya st, Moscow 101000, Russia}

\email{\newline bfeigin@gmail.com, fnklberg@gmail.com,
andrei.negut@gmail.com, leo.rybnikov@gmail.com}

\begin{abstract}
Laumon moduli spaces are certain smooth closures of the moduli spaces of
maps from the projective line to the flag variety of $GL_n$. We construct
the action of the Yangian of $\fsl_n$ in the cohomology of Laumon spaces by
certain natural correspondences. We construct the action of the affine
Yangian (two-parametric deformation of the universal enveloping algebra of the
universal central extension of
$\fsl_n[s^{\pm1},t]$) in the cohomology
of the affine version of Laumon spaces. We compute the matrix coefficients
of the generators of the affine Yangian in the fixed point basis of cohomology.
This basis is an affine analogue of the Gelfand-Tsetlin basis.
The affine analogue of the Gelfand-Tsetlin algebra surjects onto the
equivariant cohomology rings of the affine Laumon spaces. The cohomology ring
of the moduli space $\fM_{n,d}$ of torsion free sheaves on the plane, of rank
$n$ and second Chern class $d$, trivialized at infinity, is naturally embedded
into the cohomology ring of certain affine Laumon space. It is the image of
the center $Z$ of the Yangian of $\fgl_n$ naturally embedded into the affine
Yangian. In particular, the first Chern class of the determinant line bundle
on $\fM_{n,d}$ is the image of a noncommutative power sum in $Z$.
\end{abstract}
\maketitle

\section{Introduction}
\subsection{Laumon spaces and Yangians}
This note is a sequel to~\cite{fr}.
The moduli spaces $\fQ_{\ul{d}}$ were introduced by G.~Laumon in~\cite{la1}
and~\cite{la2}. They are certain partial compactifications
of the moduli spaces of degree
$\ul{d}$ based maps from $\BP^1$ to the flag variety $\CB_n$ of $GL_n$.
In~\cite{fr} we have studied the equivariant cohomology ring
$H^\bullet_{\widetilde{T}\times\BC^*}(\fQ_{\ul{d}})$ where $\widetilde T$
is a Cartan torus of $GL_n$ acting naturally on the target $\CB_n$,
and $\BC^*$ acts as ``loop rotations'' on the source $\BP^1$. The method
of~\cite{fr} was to introduce an action of $U(\fgl_n)$ on
$V=\bigoplus_{\ul{d}}H^\bullet_{\widetilde{T}\times\BC^*}(\fQ_{\ul{d}})
\otimes_{H^\bullet_{\widetilde{T}\times\BC^*}(pt)}\on{Frac}
(H^\bullet_{\widetilde{T}\times\BC^*}(pt))$ by certain natural
correspondences, and then to realize the
cohomology ring $H^\bullet_{\widetilde{T}\times\BC^*}(\fQ_{\ul{d}})$
as a certain quotient of the Gelfand-Tsetlin subalgebra
$\fG\subset U(\fgl_n)$.

In this note we adopt the following approach to the Gelfand-Tsetlin subalgebra
going back to I.~Cherednik. Namely, $\fG$ is the image of the maximal
commutative subalgebra $\fA$ of the Yangian
$Y(\fgl_n)$ (Gelfand-Tsetlin subalgebra)
under the evaluation homomorphism to $U(\fgl_n)$
(see~\cite{mbook}). Composing the evaluation homomorphism
$Y(\fgl_n)$ to $U(\fgl_n)$ with the action of $U(\fgl_n)$ on $V$ we obtain
an action of $Y(\fgl_n)$ on $V$. The main observation of this note is that
the ``new Drinfeld generators''~\cite{d} of $Y(\fsl_n)\subset Y(\fgl_n)$
act on $V$ by natural correspondences (Theorem~\ref{var}).
In fact they are very similar to
the correspondences used by M.~Varagnolo~\cite{v} to construct the action
of Yangians in the equivariant cohomology of quiver varieties.

\subsection{Affine Laumon spaces and affine Gelfand-Tsetlin bases}
There is an affine version of the Laumon spaces, namely the moduli spaces
$\CP_{\ul{d}}$ of parabolic sheaves on $\BP^1\times\BP^1$, see~\cite{fgk}.
The similar correspondences give rise to the action of the affine Yangian
$\widehat Y$ (two-parametric deformation of the universal enveloping algebra
of the universal central extension of $\fsl_n[s^{\pm1},t]$,
see~\cite{g}) on the localized equivariant cohomology
$M=\bigoplus_{\ul{d}}H^\bullet_{\widetilde{T}\times\BC^*\times\BC^*}
(\CP_{\ul{d}})
\otimes_{H^\bullet_{\widetilde{T}\times\BC^*\times\BC^*}(pt)}\on{Frac}
(H^\bullet_{\widetilde{T}\times\BC^*\times\BC^*}(pt))$ where the second
copy of $\BC^*$ acts by the loop rotation on the second copy of $\BP^1$
(Theorem~\ref{vara}). We compute explicitly the action of Drinfeld generators
of $\widehat Y$ in the fixed point basis of $M$ (Theorem~\ref{affei}).

Since the fixed point basis of $V$ corresponds to the Gelfand-Tsetlin basis
of the universal Verma module over $U(\fgl_n)$, we propose to call the
fixed point basis of $M$ the {\em affine Gelfand-Tsetlin basis}.
In particular, Conjecture~\ref{kuzn} asserts that $M$ is isomorphic to the
universal Verma module over $U(\widehat{\fgl}_n)$. Moreover, we expect
that the specialization of the affine Gelfand-Tsetlin basis gives rise
to a basis in the integrable $\widehat{\fgl}_n$-modules (which we also
propose to call the affine Gelfand-Tsetlin basis), see Conjecture~\ref{hypot}.
It seems likely that applying the Schur-Weyl functor of~\cite{g} to these
integrable modules and then going to the limit $n\to\infty$ one obtains the 
``tableaux representations'' (see~\cite{s})
of the trigonometric Cherednik algebra of type $A$.
The set of affine Gelfand-Tsetlin patterns has a structure of
$\widehat{\fsl}_n$-crystal of the integrable $\widehat{\fgl}_n$-module
(Theorem~\ref{okun}),
equivalent to that of cylindric plane partitions~\cite{t}.
We expect that the action of $\widehat Y$ on the
integrable $\widehat{\fgl}_n$-modules coincides with D.~Uglov's Yangian
action~\cite{u}.

\subsection{Cohomology ring of the Giesecker moduli space}
We prove that the maximal commutative subalgebra
of Cartan currents $\fA_{\on{aff}}\subset\widehat Y$
(the affine Gelfand-Tsetlin algebra) surjects
onto the cohomology ring of $\CP_{\ul{d}}$ (Theorem~\ref{feiryb}).
Furthermore, let $\fM_{n,d}$ denote the moduli space of torsion free sheaves
of rank $n$ and second Chern class
$d$, trivialized at infinity. The equivariant cohomology ring
$H^\bullet_{\widetilde{T}\times\BC^*\times\BC^*}(\fM_{n,d})$ is naturally
a subring of
$H^\bullet_{\widetilde{T}\times\BC^*\times\BC^*}(\CP_{d,\ldots,d})$.
There is a natural embedding $Y(\fgl_n)\hookrightarrow\widehat Y$ which
realizes the center $ZY(\fgl_n)$ as a subalgebra of $\fA_{\on{aff}}$.
This subalgebra surjects onto the cohomology ring
$H^\bullet_{\widetilde{T}\times\BC^*\times\BC^*}(\fM_{n,d})\subset
H^\bullet_{\widetilde{T}\times\BC^*\times\BC^*}(\CP_{d,\ldots,d})$.
In particular, the first Chern class of the determinant line bundle
$\D_0$ on $\fM_{n,d}$ is expressed as a certain noncommutative symmetric
function (a power sum of the second kind, see~\cite{gel})
$\Phi\in ZY(\fgl_n)$ (Theorem~\ref{xvi}).

Our results are only proved when $n>2$; however we expect them to hold for
$n=2$ as well, and it is instructive to compare them with the known results
for $n=1$. In this case $\fM_{n,d}$ is the Hilbert scheme
$\on{Hilb}^d(\BA^2)$. The first Chern class of the determinant line bundle
on $\on{Hilb}^d(\BA^2)$ was computed by M.~Lehn as a certain infinite cubic
expression (Calogero-Sutherland operator) of the
generators of the Heisenberg algebra acting by correspondences between
Hilbert schemes. In our case the role of the Heisenberg algebra is played
by $U(\widehat{\fgl}_n)$, and we were unable to express $c_1(\D_0)$ in terms
of $U(\widehat{\fgl}_n)$, but there is an explicit formula for it in terms
of $ZY(\fgl_n)$.

Finally, let us mention a trigonometric version of our note where the (affine)
Yangian is replaced with the (toroidal) affine quantum group,
and the equivariant cohomology is
replaced with the equivariant $K$-theory. This is the subject of~\cite{ts}.

\subsection{Acknowledgments}
We are obliged to A.~I.~Molev for the explanations about the Yangian
$Y(\fgl_n)$, to N.~Guay for the explanations about his affine Yangian
$\widehat Y$, and to A.~Okounkov for bringing the
references~\cite{bi},~\cite{eo},~\cite{t} to our attention.
M.~F. is grateful to A.~Braverman and A.~Kuznetsov who have been
teaching him about Laumon spaces for many years. In particular, the
construction of the transversal section $s$ in~\ref{aku} is due to
A.~Kuznetsov. Also, the assignment to describe the cohomology ring of
$\fM_{n,d}$ was given to M.F. by A.~Braverman in 2004. Thanks are due to
A.~Tsymbaliuk for the careful reading of the first draft of this note and
spotting several mistakes. We would like to thank the referee for many useful
suggestions.
B.~F. was partially supported by the grants
RFBR 08-01-00720, RFBR 05-01-01934, and NSh-6358.2006.2. M.~F. was partially
supported by the RFBR grant 09-01-00242, the Ministry of Education and
Science of Russian Federation, grant No. 2010-1.3.1-111-017-029,
and the AG Laboratory HSE, RF government grant, ag. 11.G34.31.0023.
The work of L.~R. was partially supported
by  RFBR grants 07-01-92214-CNRSL-a, 05-01-02805-CNRSL-a, 09-01-00242, 
the Science Foundation of the
SU-HSE grant 10-01-0078, and the Ministry of Education and
Science of Russian Federation, grant No. 2010-1.3.1-111-017-029.
He gratefully acknowledges the support from Deligne 2004 Balzan prize in
mathematics. A.~N. would like to thank the Institut des Hautes Etudes
Scientifiques for their hospitality during the time when this paper was
written, and LEA Math-Mode for their support.

\section{Laumon spaces and $\fsl_n$-Yangians}
\label{odin}

\subsection{Laumon spaces}
We recall the setup of ~\cite{fr}. Let $\bC$ be a smooth
projective curve of genus zero. We fix a coordinate $z$ on $\bC$,
and consider the action of $\BC^*$ on $\bC$ such that
$v(z)=v^{-2}z$. We have $\bC^{\BC^*}=\{0,\infty\}$.

We consider an $n$-dimensional vector space $W$ with a basis
$w_1,\ldots,w_n$. This defines a Cartan torus $T\subset
G=GL_n\subset Aut(W)$. We also consider its $2^n$-fold cover,
the bigger torus $\widetilde{T}$, acting on $W$ as follows: for
$\widetilde{T}\ni\ul{t}=(t_1,\ldots,t_n)$ we have
$\ul{t}(w_i)=t_i^2w_i$. We denote by $\CB$ the flag variety of
$G$.

Given an $(n-1)$-tuple of nonnegative integers
$\ul{d}=(d_1,\ldots,d_{n-1})$, we consider the Laumon's
quasiflags' space $\CQ_{\ul{d}}$, see ~\cite{la2}, ~4.2. It is the
moduli space of flags of locally free subsheaves
$$0\subset\CW_1\subset\ldots\subset\CW_{n-1}\subset\CW=W\otimes\CO_\bC$$
such that $\on{rank}(\CW_k)=k$, and $\deg(\CW_k)=-d_k$.

It is known to be a smooth projective variety of dimension
$2d_1+\ldots+2d_{n-1}+\dim\CB$, see ~\cite{la1}, ~2.10.

$\fQ_{\ul{d}}\subset\CQ_{\ul{d}}$ (quasiflags based at
$\infty\in\bC$) formed by the flags
$$0\subset\CW_1\subset\ldots\subset\CW_{n-1}\subset\CW=W\otimes\CO_\bC$$
such that $\CW_i\subset\CW$ is a vector subbundle in a
neighbourhood of $\infty\in\bC$, and the fiber of $\CW_i$ at
$\infty$ equals the span $\langle w_1,\ldots,w_i\rangle\subset W$.

It is known to be a smooth quasiprojective variety of dimension
$2d_1+\ldots+2d_{n-1}$.

\subsection{Fixed points}
\label{fixed points} The group $G\times\BC^*$ acts naturally on
$\CQ_{\ul{d}}$, and the group $\widetilde{T}\times\BC^*$ acts
naturally on $\fQ_{\ul{d}}$. The set of fixed points of
$\widetilde{T}\times\BC^*$ on $\fQ_{\ul{d}}$ is finite; we recall
its description from ~\cite{fk}, ~2.11.

Let $\widetilde{\ul{d}}$ be a collection of nonnegative integers
$(d_{ij}),\ i\geq j$, such that $d_i=\sum_{j=1}^id_{ij}$, and for
$i\geq k\geq j$ we have $d_{kj}\geq d_{ij}$. Abusing notation we
denote by $\widetilde{\ul{d}}$ the corresponding
$\widetilde{T}\times\BC^*$-fixed point in $\fQ_{\ul{d}}$:

$\CW_1=\CO_\bC(-d_{11}\cdot0)w_1,$

$\CW_2=\CO_\bC(-d_{21}\cdot0)w_1\oplus\CO_\bC(-d_{22}\cdot0)w_2,$

$\ldots\ \ldots\ \ldots\ ,$

$\CW_{n-1}=\CO_\bC(-d_{n-1,1}\cdot0)w_1\oplus\CO_\bC(-d_{n-1,2}\cdot0)w_2
\oplus\ldots\oplus\CO_\bC(-d_{n-1,n-1}\cdot0)w_{n-1}.$

\subsection{Correspondences}
\label{classic}
For $i\in\{1,\ldots,n-1\}$, and $\ul{d}=(d_1,\ldots,d_{n-1})$, we
set $\ul{d}+i:=(d_1,\ldots,d_i+1,\ldots,d_{n-1})$. We have a
correspondence $\CE_{\ul{d},i}\subset\CQ_{\ul{d}}\times
\CQ_{\ul{d}+i}$ formed by the pairs $(\CW_\bullet,\CW'_\bullet)$
such that for $j\ne i$ we have $\CW_j=\CW'_j$, and
$\CW'_i\subset\CW_i$, see ~\cite{fk}, ~3.1. In other words,
$\CE_{\ul{d},i}$ is the moduli space of flags of locally free
sheaves
$$0\subset\CW_1\subset\ldots\CW_{i-1}\subset\CW'_i\subset\CW_i\subset
\CW_{i+1}\ldots\subset\CW_{n-1}\subset\CW$$ such that
$\on{rank}(\CW_k)=k$, and $\deg(\CW_k)=-d_k$, while
$\on{rank}(\CW'_i)=i$, and $\deg(\CW'_i)=-d_i-1$.

According to ~\cite{la1}, ~2.10, $\CE_{\ul{d},i}$ is a smooth
projective algebraic variety of dimension
$2d_1+\ldots+2d_{n-1}+\dim\CB+1$.

We denote by $\bp$ (resp. $\bq$) the natural projection
$\CE_{\ul{d},i}\to\CQ_{\ul{d}}$ (resp.
$\CE_{\ul{d},i}\to\CQ_{\ul{d}+i}$). We also have a map $\bfr:\
\CE_{\ul{d},i}\to\bC,$
$$(0\subset\CW_1\subset\ldots\CW_{i-1}\subset\CW'_i\subset\CW_i\subset
\CW_{i+1}\ldots\subset\CW_{n-1}\subset\CW)\mapsto\on{supp}(\CW_i/\CW'_i).$$

The correspondence $\CE_{\ul{d},i}$ comes equipped with a natural
line bundle $\CL_i$ whose fiber at a point
$$(0\subset\CW_1\subset\ldots\CW_{i-1}\subset\CW'_i\subset\CW_i\subset
\CW_{i+1}\ldots\subset\CW_{n-1}\subset\CW)$$ equals
$\Gamma(\bC,\CW_i/\CW'_i)$.

Finally, we have a transposed correspondence
$^\sT\CE_{\ul{d},i}\subset \CQ_{\ul{d}+i}\times\CQ_{\ul{d}}$.

Restricting to $\fQ_{\ul{d}}\subset\CQ_{\ul{d}}$ we obtain the
correspondence
$\fE_{\ul{d},i}\subset\fQ_{\ul{d}}\times\fQ_{\ul{d}+i}$ together
with line bundle $\fL_i$ and the natural maps $\bp:\
\fE_{\ul{d},i}\to\fQ_{\ul{d}},\ \bq:\
\fE_{\ul{d},i}\to\fQ_{\ul{d}+i},\ \bfr:\
\fE_{\ul{d},i}\to\bC-\infty$. We also have a transposed
correspondence $^\sT\fE_{\ul{d},i}\subset
\fQ_{\ul{d}+i}\times\fQ_{\ul{d}}$. It is a smooth quasiprojective
variety of dimension $2d_1+\ldots+2d_{n-1}+1$.

\subsection{Equivariant cohomology}
We denote by ${}'V$ the direct sum of equivariant (complexified)
cohomology:
${}'V=\oplus_{\ul{d}}H^\bullet_{\widetilde{T}\times\BC^*}(\fQ_{\ul{d}})$.
It is a module over
$H^\bullet_{\widetilde{T}\times\BC^*}(pt)=\BC[\ft\oplus\BC]=
\BC[x_1,\ldots,x_n,\hbar]$. Here $\ft\oplus\BC$ is the
Lie algebra of $\widetilde{T}\times\BC^*$. We define $\hbar$ as twice
the positive generator of $H^2_{\BC^*}(pt,\BZ)$. Similarly, we define
$x_i\in H^2_{\widetilde{T}}(pt,\BZ)$ in terms of the corresponding
one-parametric subgroup.
We define $V=\ {}'V\otimes_{H^\bullet_{\widetilde{T}\times\BC^*}(pt)}
\on{Frac}(H^\bullet_{\widetilde{T}\times\BC^*}(pt))$.

We have an evident grading $$V=\oplus_{\ul{d}}V_{\ul{d}},\
V_{\ul{d}}=H^\bullet_{\widetilde{T}\times\BC^*}(\fQ_{\ul{d}})
\otimes_{H^\bullet_{\widetilde{T}\times\BC^*}(pt)}
\on{Frac}(H^\bullet_{\widetilde{T}\times\BC^*}(pt)).$$

\subsection{Universal Verma module}
We denote by
$\fU$ the universal enveloping algebra of $\fgl_n$ over the field
$\BC(\ft\oplus\BC)$. For $1\leq j,k\leq n$ we denote by
$E_{jk}\in\fgl_n\subset\fU$ the usual elementary matrix.
The standard Chevalley generators are expressed as follows:
$$\fe_i:=E_{i+1,i},\
\ff_i:=E_{i,i+1},\ \fh_i:=E_{i+1,i+1}-E_{ii}$$
(note that $\fe_i$ is represented by a {\em lower} triangular matrix).
Note also that $\fU$ is generated by $E_{ii},\ 1\leq i\leq n,\
E_{i,i+1},E_{i+1,i},\ 1\leq i\leq n-1$.
We denote by $\fU_{\leq0}$ the subalgebra of $\fU$ generated by
$E_{ii},\ 1\leq i\leq n,\ E_{i,i+1},\ 1\leq i\leq n-1$.
It acts on the field $\BC(\ft\oplus\BC)$ as follows:
$E_{i,i+1}$ acts trivially for any $1\leq i\leq n-1$, and
$E_{ii}$ acts by multiplication by
$\hbar^{-1}x_i+i-1$. We define the {\em universal Verma module}
$\fV$ over $\fU$ as $\fU\otimes_{\fU_{\leq0}}\BC(\ft\oplus\BC)$.
The universal Verma module $\fV$ is an irreducible $\fU$-module.

\subsection{The action of generators}
\label{operators} The grading and the correspondences
$^\sT\fE_{\ul{d},i},\fE_{\ul{d},i}$ give rise to the following
operators on $V$ (note that though $\bp$ is not proper, $\bp_*$ is
well defined on the localized equivariant cohomology due to the
finiteness of the fixed point sets):

$E_{ii}=\hbar^{-1}x_i+d_{i-1}-d_i+i-1:\ V_{\ul{d}}\to
V_{\ul{d}}$;

$\fh_i=\hbar^{-1}(x_{i+1}-x_i)+2d_i-d_{i-1}-d_{i+1}+1:\ V_{\ul{d}}\to
V_{\ul{d}}$;

$\ff_i=E_{i,i+1}=\bp_*\bq^*:\ V_{\ul{d}}\to V_{\ul{d}-i}$;

$\fe_i=E_{i+1,i}=-\bq_*\bp^*:\
V_{\ul{d}}\to V_{\ul{d}+i}$.

The following theorem is Theorem~2.7 of~\cite{fr}.

\begin{thm}
\label{brav}
The operators $\fe_i=E_{i+1,i},E_{ii},\ff_i=E_{i,i+1}$ on $V$ defined
in ~\ref{operators} satisfy the relations in $\fU$, i.e.
they give rise to the action of $\fU$ on $V$.
There is a unique isomorphism $\Psi$ of $\fU$-modules
$V$ and $\fV$ carrying $1\in H^0_{\widetilde{T}\times\BC^*}(\fQ_0)\subset V$
to the lowest weight vector $1\in\BC(\ft\oplus\BC)\subset\fV$.
\end{thm}

\subsection{Gelfand-Tsetlin basis of the universal Verma module}
\label{classical}
We will follow the notations of~\cite{m} on the Gelfand-Tsetlin bases in
representations of $\fgl_n$.
To a collection $\widetilde{\ul{d}}=(d_{ij}),\ n-1\geq i\geq j$ we associate
a {\em Gelfand-Tsetlin pattern}
$\Lambda=\Lambda(\widetilde{\ul{d}}):=(\lambda_{ij}),\ n\geq i\geq j$
as follows:
$\lambda_{nj}:=\hbar^{-1}x_j+j-1,\ n\geq j\geq 1;\
\lambda_{ij}:=\hbar^{-1}x_j+j-1-d_{ij},\ n-1\geq i\geq j\geq1$.
Now we define $\xi_{\widetilde{\ul{d}}}=\xi_\Lambda\in\fV$ by the formulas
(2.9)--(2.11) of {\em loc. cit.} (where $\xi=\xi_0=1\in\fV$).
According to Theorem~2.7 of {\em loc. cit.}, the set
$\{\xi_{\widetilde{\ul{d}}}\}$ (over all collections $\widetilde{\ul{d}}$)
forms a basis of $\fV$.

According to the Thomason localization theorem,
restriction to the $\widetilde{T}\times\BC^*$-fixed
point set induces an isomorphism
$$H^\bullet_{\widetilde{T}\times\BC^*}(\fQ_{\ul{d}})
\otimes_{H^\bullet_{\widetilde{T}\times\BC^*}(pt)}
\on{Frac}(H^\bullet_{\widetilde{T}\times\BC^*}(pt))\iso$$
$$\iso H^\bullet_{\widetilde{T}\times\BC^*}(\fQ_{\ul{d}}^{\widetilde{T}\times\BC^*})
\otimes_{H^\bullet_{\widetilde{T}\times\BC^*}(pt)}
\on{Frac}(H^\bullet_{\widetilde{T}\times\BC^*}(pt))$$

The fundamental cycles $[\widetilde{\ul{d}}]$ of the
$\widetilde{T}\times\BC^*$-fixed points $\widetilde{\ul{d}}$ (see
~\ref{fixed points}) form a basis in
$\oplus_{\ul{d}}H^\bullet_{\widetilde{T}\times\BC^*}
(\fQ_{\ul{d}}^{\widetilde{T}\times\BC^*})
\otimes_{H^\bullet_{\widetilde{T}\times\BC^*}(pt)}\on{Frac}
(H^\bullet_{\widetilde{T}\times\BC^*}(pt))$. The embedding of a point
$\widetilde{\ul{d}}$ into $\fQ_{\ul{d}}$ is a proper morphism, so the
direct image in the equivariant cohomology is well defined, and we will
denote by $[\widetilde{\ul{d}}]\in V_{\ul{d}}$ the direct image of the
fundamental cycle of the point $\widetilde{\ul{d}}$. The set
$\{[\widetilde{\ul{d}}]\}$ forms a basis of $V$.

The following theorem is~Theorem~2.11 and~Proposition~2.9 of~\cite{fr}, cf. also~\cite{ne}~8.2.

\begin{thm}
\label{feigin}
a) The isomorphism $\Psi:\ V\iso\fV$ of Theorem~\ref{brav} takes
$[\widetilde{\ul{d}}]$ to $(-\hbar)^{|\ul{d}|}\xi_{\widetilde{\ul{d}}}$
where $|\ul{d}|=d_1+\ldots+d_{n-1}$.

b) The matrix coefficients of the operators $\fe_i,\ff_i$ in the basis
$\{[\widetilde{\ul{d}}]\}$ are as follows:
$$\fe_{i[\widetilde{\ul{d}},\widetilde{\ul{d}}{}']}=
-\hbar^{-1}\prod_{j\ne k\leq i}(x_j-x_k+(d_{i,k}-d_{i,j})\hbar)^{-1}
\prod_{k\leq i-1}(x_j-x_k+(d_{i-1,k}-d_{i,j})\hbar)$$ if
$d'_{i,j}=d_{i,j}+1$ for certain $j\leq i$;

$$\ff_{i[\widetilde{\ul{d}},\widetilde{\ul{d}}{}']}=
\hbar^{-1}\prod_{j\ne k\leq i}(x_k-x_j+(d_{i,j}-d_{i,k})\hbar)^{-1}
\prod_{k\leq i+1}(x_k-x_j+(d_{i,j}-d_{i+1,k})\hbar)$$ if
$d'_{i,j}=d_{i,j}-1$ for certain $j\leq i$;

All the other matrix coefficients of $\fe_i,\ff_i$ vanish.
\end{thm}

\subsection{Yangian of $\fsl_n$}
\label{yang fin}
Let $(a_{kl})_{1\leq k,l\leq n-1}=A_{n-1}$ stand for the Cartan matrix of
$\fsl_n$. The Yangian $Y(\fsl_n)$ is the free $\BC[\hbar]$-algebra generated
by $\bx_{k,r}^\pm,\bh_{k,r},\ 1\leq k\leq n-1,\ r\in\BN$, with the following
relations:

\begin{equation}
\label{11}
[\bh_{k,r},\bh_{l,s}]=0,\ [\bh_{k,0},\bx_{l,s}^\pm]=\pm a_{kl}\bx_{l,s}^\pm,
\end{equation}

\begin{equation}
\label{12}
2[\bh_{k,r+1},\bx_{l,s}^\pm]-2[\bh_{k,r},\bx_{l,s+1}^\pm]=
\pm\hbar a_{kl}(\bh_{k,r}\bx_{l,s}^\pm+\bx_{l,s}^\pm\bh_{k,r}),
\end{equation}

\begin{equation}
\label{13}
[\bx^+_{k,r},\bx^-_{l,s}]=\delta_{kl}\bh_{k,r+s},
\end{equation}

\begin{equation}
\label{14}
2[\bx_{k,r+1}^\pm,\bx_{l,s}^\pm]-2[\bx_{k,r}^\pm,\bx_{l,s+1}^\pm]=
\pm\hbar a_{kl}(\bx_{k,r}^\pm\bx_{l,s}^\pm+\bx_{l,s}^\pm\bx_{k,r}^\pm),
\end{equation}

\begin{equation}
\label{15}
[\bx_{k,r}^\pm,[\bx_{k,p}^\pm,\bx_{l,s}^\pm]]+
[\bx_{k,p}^\pm,[\bx_{k,r}^\pm,\bx_{l,s}^\pm]]=0,\
k=l\pm1,\ \forall p,r,s\in\BN.
\end{equation}

For a formal variable $u$ we introduce the generating series
$\bh_k(u):=1+\sum_{r=0}^\infty\bh_{k,r}\hbar^{-r}u^{-r-1};\
\bx_k^\pm(u):=\sum_{r=0}^\infty\bx_{k,r}^\pm\hbar^{-r}u^{-r-1}$. We can then rewrite the equations~(\ref{12},\ref{14}) in the following form

\begin{equation}
\label{12'}
\partial_u\partial_v\bh_k(u)\bx_l^\pm(v)(2u-2v\mp a_{kl})=-\partial_u\partial_v\bx_l^\pm(v)\bh_k(u)(2v-2u\mp a_{kl})
\end{equation}

\begin{equation}
\label{14'}
\partial_u\partial_v\bx_k^\pm(u)\bx_l^\pm(v)(2u-2v\mp a_{kl})=-\partial_u\partial_v\bx_l^\pm(v)\bx_k^\pm(u)(2v-2u\mp a_{kl})
\end{equation}

\subsection{The action of Yangian generators}
\label{yang cor}
For any $0\leq i\leq n$ we will denote by $\ul{\CW}{}_i$ the tautological
$i$-dimensional vector bundle on $\fQ_{\ul{d}}\times\bC$. Let $\pi:\fQ_{\ul{d}}\times (\bC\ \backslash \ \infty) \rightarrow \fQ_{\ul{d}}$ denote the standard projection. Let $q$ stand for the character of $\widetilde{T}\times\BC^*:\
(\ul{t},v)\mapsto v^2$.
We define the line bundle $\CL'_k:=q^{\frac{1-k}{2}}\CL_k$ on the
correspondence $\CE_{\ul{d},k}$, that is $\CL'_k$ and $\CL_k$ are
isomorphic as line bundles but the equivariant structure of $\CL'_k$ is obtained
from the equivariant structure of $\CL_k$ by the twist by the
character $q^{\frac{1-k}{2}}$. Finally, for a vector bundle $\CV$ and a 
formal variable $x$ we denote by $c(\CV,x)$ the Chern polynomial. 

Consider the operators:




$$
\ba_m(u)=u^m\cdot \bp_*(c(\pi_*(\ul{\CW}_m|_{\bC \backslash \infty}), (-u\hbar)^{-1})\cdot\bq^*):\ V_{\ul{d}}\to V_{\ul{d}}[[u^{-1}]][u]
$$

\begin{equation}
\label{ddvas}
\bx_{k,r}^+:=\bp_*(c_1(\CL'_k)^r\cdot\bq^*):\ V_{\ul{d}}\to V_{\ul{d}-k}
\end{equation}

\begin{equation}
\label{ttris}
\bx_{k,r}^-:=-\bq_*(c_1(\CL'_k)^r\cdot\bp^*):\ V_{\ul{d}}\to V_{\ul{d}+k}
\end{equation}

We consider the following generating series of operators on $V$:

\begin{multline}
\label{raz}
\bh_k(u)=1+\sum_{r=0}^\infty\bh_{k,r}\hbar^{-r}u^{-r-1}:=\\=
\ba_k(u+\frac{k+1}{2})^{-1}\ba_k(u+\frac{k-1}{2})^{-1}\ba_{k-1}(u+\frac{k-1}{2})\ba_{k+1}(u+\frac{k+1}{2}):\\
V_{\ul{d}}\to V_{\ul{d}}[[u^{-1}]];
\end{multline}

\begin{equation}
\label{dvas}
\bx_{k}^+(u)=\sum_{r=0}^\infty\bx_{k,r}^+\hbar^{-r}u^{-r-1}:\
V_{\ul{d}}\to V_{\ul{d}-k}[[u^{-1}]]
\end{equation}

\begin{equation}
\label{tris}
\bx_{k}^-(u)=\sum_{r=0}^\infty\bx_{k,r}^-\hbar^{-r}u^{-r-1}:\
V_{\ul{d}}\to V_{\ul{d}+k}[[u^{-1}]]
\end{equation}

\begin{thm}
\label{var}
The operators $\bh_{k,r},\bx_{k,r}^\pm$ on $V$ defined in~\ref{yang cor}
satisfy the relations in $Y(\fsl_n)$, i.e. they give rise to the action of
$Y(\fsl_n)$ on $V$.
\end{thm}

\proof We identify $V$ with the universal Verma module $\fV$ via the
isomorphism $\Psi$ of Theorem~\ref{brav}. We consider the operators
$\bA_m(u),\bB_m(u),\bC_m(u):\ \fV\to\fV[u],\ 1\leq m\leq n-1$,
and also $\bA_0(u)=1$, and $\bA_n(u)$, introduced
in~\cite{mbook}~5.3.
The explicit formula for the (diagonal) action of $\bA_m(u)$ in the
Gelfand-Tsetlin basis $\{\xi_{\widetilde{\ul{d}}}\}$ of $\fV$ is given in
Theorem~5.3.4 of~\cite{mbook}. It reads
\begin{equation}
\label{239}
\bA_m(u)\xi_{\widetilde{\ul{d}}}=(u+\hbar^{-1}x_1-d_{m1})\ldots
(u+\hbar^{-1}x_m-d_{mm})\xi_{\widetilde{\ul{d}}}
\end{equation}
The explicit formula for the (diagonal) action
of $\ba_m(u)$ in the fixed point basis $\{[\widetilde{\ul{d}}]\}$ of $V$
is given in the proof of Theorem~3.5 of~\cite{fr}. Comparing the two formulas
we see that the eigenvalues of $\bA_m(u)$ in the Gelfand-Tsetlin basis
coincide with the eigenvalues of $\ba_m(u)$ in the fixed point basis.
Now Theorem~\ref{feigin} implies that $\Psi$ takes $\ba_m(u)$ to $\bA_m(u)$.

Recall that there is another ``RTT'' presentation of $Y(\fsl_n)$,
see~\cite{mbook}. It is related to the ``new Drinfeld
presentation''~\ref{yang fin} by the Drinfeld isomorphism
(see~\cite{mbook}~5.3~and~3.1.8):

\begin{equation}
\label{iso1}
\bh_k(u)=
\bA_k(u+\frac{k+1}{2})^{-1}\bA_k(u+\frac{k-1}{2})^{-1}\bA_{k-1}(u+\frac{k-1}{2})\bA_{k+1}(u+\frac{k+1}{2})
\end{equation}

\begin{equation}
\label{iso2}
\bx_k^+(u)=\bA_k(u+\frac{k-1}{2})^{-1}\bB_k(u+\frac{k-1}{2})
\end{equation}

\begin{equation}
\label{iso3}
\bx_k^-(u)=\bC_k(u+\frac{k-1}{2})\bA_k(u+\frac{k-1}{2})^{-1}
\end{equation}

According to~\cite{mbook}, the operators
$\bA_k(u),\bB_k(u),\bC_k(u)$ arise from
an action of $Y(\fsl_n)$ in the RTT presentation, so the LHS of
equations~(\ref{iso1},\ref{iso2},\ref{iso3}) do satisfy the
relations~(\ref{11}--\ref{15}). So in order to prove the theorem,
it remains to check that the isomorphism $\Psi$ takes the generating series
of the LHS of~(\ref{dvas}) (resp. of~(\ref{tris})) to the LHS of~(\ref{iso2})
(resp. of~(\ref{iso3})). We consider the case of $\bx_k^-(u)$, the case of
$\bx_k^+(u)$ being absolutely similar.

The character of $\widetilde{T}\times\BC^*$ in the fiber of the line bundle
$\CL_i$ at a point $(\widetilde{\ul{d}},\widetilde{\ul{d}}{}')
\in\fE_{\ul{d},i}$ equals $-x_j+d_{ij}\hbar$ if $d_{ij}'=d_{ij}+1$ for certain
$j\leq i$. It follows from Theorem~\ref{feigin} b) that the matrix coefficients
of $\bx_{i,r}^-$ in the basis $\{[\widetilde{\ul{d}}]\}$ are given by

\begin{multline}
\label{mole}
\bx_{i,r[\widetilde{\ul{d}},\widetilde{\ul{d}}{}']}^-=
(-x_j+(d_{ij}+\frac{1-i}{2})\hbar)^r
\fe_{i[\widetilde{\ul{d}},\widetilde{\ul{d}}{}']}=
=-\hbar^{-1}(-x_j+(d_{ij}+\frac{1-i}{2})\hbar)^r\cdot\\
\prod_{j\ne k\leq i}(x_j-x_k+(d_{i,k}-d_{i,j})\hbar)^{-1}
\prod_{k\leq i-1}(x_j-x_k+(d_{i-1,k}-d_{i,j})\hbar)
\end{multline}
if $d'_{i,j}=d_{i,j}+1$ for certain $j\leq i$, and all the other matrix
coefficients vanish.

On the other hand, the matrix coefficients of $\bC_i(u)\bA_i(u)^{-1}$
in the Gelfand-Tsetlin basis are computed in Theorem~5.3.4
of~\cite{mbook}. Namely, if $d'_{i,j}=d_{i,j}+1$ for certain $j\leq i$, then
specializing $u=d_{ij}-\hbar^{-1}x_j$ we have
\begin{equation}
\label{240}
\bC_i(d_{ij}-\hbar^{-1}x_j)\xi_{\widetilde{\ul{d}}}=\prod_{k=1}^{i-1}
(\hbar^{-1}(x_k-x_j)-d_{i-1,k}+d_{ij})\xi_{\widetilde{\ul{d}}{}'}
\end{equation}
Since $\deg \bC_i(u)=i-1$, we can find the matrix coefficients of $\bC_i(u)$,
and then of $\bC_i(u+\frac{i-1}{2})\bA_i(u+\frac{i-1}{2})^{-1}$
by the Lagrange interpolation.
The resulting formula coincides with the negative of~(\ref{mole}).
Since the isomorphism $\Psi$ takes
$[\widetilde{\ul{d}}]$ to $(-\hbar)^{|\ul{d}|}\xi_{\widetilde{\ul{d}}}$,
the coincidence of the two formulas completes the proof of the
theorem. \qed

\subsection{Modified generators}
\label{quotient}
We formulate a corollary which will be used in Section~\ref{psay}.
For any $0\leq m<i\leq n$ we will denote by $\ul{\CW}{}_{mi}$ the
quotient $\ul{\CW}{}_i/\ul{\CW}{}_m$ of the tautological
vector bundles on $\fQ_{\ul{d}}\times\bC$. Similarly to the above, we introduce the generating series:

$$
\ba_{mi}(u)=u^{i-m}\cdot\bp_*(c(\pi_*(\ul{\CW}_{mi}|_{\bC \backslash \infty}), 
(-u\hbar)^{-1})\cdot\bq^*):\ V_{\ul{d}}\to V_{\ul{d}}[[u^{-1}]][u]
$$

\begin{cor}
\label{faktor}
The operator $\bh_i(u)=\ba_{mi}(u+\frac{i-1}{2})^{-1}\ba_{mi}(u+\frac{i+1}{2})^{-1}\ba_{m,i-1}(u+\frac{i-1}{2})
\ba_{m,i+1}(u+\frac{i+1}{2})$ does not depend on $m$, for any $m<i$.
\end{cor}

\proof Let us denote the RHS of the corollary by $\bh_{mi}(u)$. We have
to prove that $\bh_{mi}(u)=\bh_i(u)$. Due to the relation~(\ref{13}) we
have to check $\bh_{mi,r+s}=[\bx_{i,r}^+,\bx_{i,s}^-]$. This is done
by reduction to the moduli stack $\fZ_k$ introduced in section~3.11
of~\cite{bf}, absolutely similarly to {\em loc. cit.} 

More precisely, we consider a $k$-dimensional vector space with a basis
$\fw_1,\ldots,\fw_k$, and a torus $\fT$ acting on $\fw_l$ by the
character $\tau^2_l$. Let $\fZ_k$ be the moduli stack of flags of
coherent sheaves $\fW_1\subset\ldots\subset\fW_k$ on $\bC$ locally
free at $\infty\in\bC$, equipped with compatible trivializations
$\fW_l|_\infty=\langle\fw_1,\ldots,\fw_l\rangle$. Note that
$\fZ_k$ has connected components numbered by the degrees of $\fW_l$. 
Absolutely similarly to~\ref{classic} we
introduce the correspondences between various connected
components. They are equipped with line bundles similar to~\ref{classic}.
This gives rise to the operators
$^\fZ\bx^{\pm}_{i,r},\ ^\fZ\ba_i(u),\ ^\fZ\bh_{i,r},\ 
1\leq i\leq k-1$, on the localized $\fT\times\BC^*$-equivariant cohomology of 
$\fZ_k$. 

For $k=n-m$, we have a map
$\fz_k:\ \fQ_{\ul{d}}\to\fZ_k,\ (\CW_\bullet)\mapsto
(\CW_{m+1}/\CW_m\subset\ldots\subset\CW_n/\CW_m)$. 
The argument of~\cite[Sections~3.10,3.11]{bf} deduces the desired relation
$\bh_{mi,r+s}=[\bx_{i,r}^+,\bx_{i,s}^-]$ from a certain weak form of the
identity $^\fZ\bh_{i-m,r+s}=[\ ^\fZ\bx_{i-m,r}^+,\ ^\fZ\bx_{i-m,s}^-]$. The (weak 
form of the) latter relation is in turn deduced in {\em loc. cit.} from
$\bh_{i,r+s}=[\bx_{i,r}^+,\bx_{i,s}^-]$.

Alternatively, the corollary can be proved by a direct calculation which is
a rational version of the trigonometric calculation in the proof 
of~\cite[Corollary~2.16]{ts}.
\qed

\section{Parabolic sheaves and affine Yangians}
\label{psay}
In this section we generalize the previous results to the
affine setting.

\subsection{Parabolic sheaves}
\label{PS}
We recall the setup of section~3 of~\cite{bf}. Let $\bX$ be another smooth
projective curve of genus zero. We fix a coordinate $y$ on $\bX$,
and consider the action of $\BC^*$ on $\bX$ such that
$c(x)=c^{-2}x$. We have $\bX^{\BC^*}=\{0_\bX,\infty_\bX\}$. Let
$\bS$ denote the product surface $\bC\times\bX$. Let $\bD_\infty$
denote the divisor $\bC\times\infty_\bX\cup\infty_\bC\times\bX$.
Let $\bD_0$ denote the divisor $\bC\times0_\bX$.

Given an $n$-tuple of nonnegative integers
$\ul{d}=(d_0,\ldots,d_{n-1})$, we say that a {\em parabolic sheaf}
$\CF_\bullet$ of degree $\ul{d}$ is an infinite flag of torsion
free coherent sheaves of rank $n$ on $\bS:\
\ldots\subset\CF_{-1}\subset\CF_0\subset\CF_1\subset\ldots$ such
that:

(a) $\CF_{k+n}=\CF_k(\bD_0)$ for any $k$;

(b) $ch_1(\CF_k)=k[\bD_0]$ for any $k$: the first Chern classes
are proportional to the fundamental class of $\bD_0$;

(c) $ch_2(\CF_k)=d_i$ for $i\equiv k\pmod{n}$;

(d) $\CF_0$ is locally free at $\bD_\infty$ and trivialized at
$\bD_\infty:\ \CF_0|_{\bD_\infty}=W\otimes\CO_{\bD_\infty}$;

(e) For $-n\leq k\leq0$ the sheaf $\CF_k$ is locally free at
$\bD_\infty$, and the quotient sheaves $\CF_k/\CF_{-n},\
\CF_0/\CF_k$ (both supported at $\bD_0=\bC\times0_\bX\subset\bS$)
are both locally free at the point $\infty_\bC\times0_\bX$;
moreover, the local sections of $\CF_k|_{\infty_\bC\times \bX}$
are those sections of $\CF_0|_{\infty_\bC\times
\bX}=W\otimes\CO_\bX$ which take value in $\langle
w_1,\ldots,w_{n+k}\rangle\subset W$ at $0_\bX\in \bX$.

\medskip

The fine moduli space
$\CP_{\ul{d}}$ of degree $\ul{d}$ parabolic sheaves exists and is
a smooth connected quasiprojective variety of dimension
$2d_0+\ldots+2d_{n-1}$.

\subsection{Fixed points}
\label{fp}
The group $\widetilde{T}\times\BC^*\times\BC^*$ acts naturally on
$\CP_{\ul{d}}$, and its fixed point set is finite. In order to describe it,
we recall the well known description of the fixed point set of
$\BC^*\times\BC^*$ on the Hilbert scheme of
$(\bC-\infty_\bC)\times(\bX-\infty_\bX)$. Namely, the fixed points are
parametrized by the Young diagrams, and for a diagram
$\lambda=(\lambda_0\geq\lambda_1\geq\ldots)$ (where $\lambda_N=0$ for $N\gg0$)
the corresponding fixed point is the ideal
$J_\lambda=\BC[z]\cdot(\BC y^0z^{\lambda_0}\oplus
\BC y^1z^{\lambda_1}\oplus\ldots)$. We will view $J_\lambda$ as an ideal in
$\CO_{\bC\times\bX}$ coinciding with $\CO_{\bC\times\bX}$ in a neighbourhood
of infinity.

We say $\lambda\supset\mu$ if $\lambda_i\geq\mu_i$ for any $i\geq0$.
We say $\lambda\widetilde\supset\mu$ if $\lambda_i\geq\mu_{i+1}$ for
any $i\geq0$.

We consider a collection $\blambda=(\lambda^{kl})_{1\leq k,l\leq n}$
of Young diagrams satisfying the following inequalities:
\begin{multline}
\label{pool}
\lambda^{11}\supset\lambda^{21}\supset\ldots\supset\lambda^{n1}
\widetilde\supset\lambda^{11};\
\lambda^{22}\supset\lambda^{32}\supset\ldots\supset\lambda^{12}
\widetilde\supset\lambda^{22};\
\ldots;\\
\lambda^{nn}\supset\lambda^{1n}\supset\ldots\supset\lambda^{n-1,n}
\widetilde\supset\lambda^{nn}
\end{multline}
We set $d_k(\blambda)=\sum_{l=1}^n|\lambda^{kl}|$, and
$\ul{d}(\blambda)=(d_0(\blambda):=d_n(\blambda),\ldots,d_{n-1}(\blambda))$.

Given such a collection $\blambda$ we define a parabolic sheaf
$\CF_\bullet=\CF_\bullet(\blambda)$, or just $\blambda$ by an abuse of
notation, as follows: for $1\leq k\leq n$ we set
\begin{equation}
\label{swim}
\CF_{k-n}=\bigoplus_{1\leq l\leq k} J_{\lambda^{kl}}w_l\oplus
\bigoplus_{k<l\leq n}J_{\lambda^{kl}}(-\bD_0)w_l
\end{equation}

\begin{lem}
\label{evi}
The correspondence $\blambda\mapsto\CF_\bullet(\blambda)$ is a bijection
between the set of collections $\blambda$ satisfying~(\ref{pool})
such that $\ul{d}(\blambda)=\ul{d}$,
and the set of $\widetilde{T}\times\BC^*\times\BC^*$-fixed points in
$\CP_{\ul{d}}$.
\end{lem}

\proof It is easy to see that flags of the form \eqref{swim} are torus fixed points. Note that the condition $\lambda^{kk} \widetilde\subset \lambda^{k-1,k}$ appears because of the twists by $\bD_0$ in \eqref{swim}. Conversely, any torus fixed point splits up as a sum of rank 1 torsion free sheaves:

$$
\CF_{k-n} = \bigoplus_l T_{kl} w_l,
$$
just as in the well-known case of rank $n$ torsion free sheaves on a surface. The rank 1 torsion free sheaves $T_{kl}$ are just twisted ideal sheaves: $T_{kl} = J_{\lambda^{kl}}$ for $l\leq k$ and $T_{kl} = J_{\lambda^{kl}}(-\bD_0)$ for $l>k$, as follows from the framing and $c_1$ conditions. Finally, the fact that the partitions $\lambda^{kl}$ have to satisfy \eqref{pool} is simply equivalent to the conditions $\CF_{k-n} \subset \CF_{k-1-n}$. \qed

\subsection{Parabolic sheaves as orbifold sheaves} \label{realization}
We will now introduce a different realization of parabolic sheaves, and
another parameterization of the fixed point set which is very closely related
to this new realization. We first learned of this construction from A. Okounkov,
though it is already present in the work of Biswas~\cite{bi}. Let
$\sigma:\bC\times \bX\rightarrow \bC\times \bX$ denote the map
$\sigma(z,y)=(z,y^n)$, and let $G=\BZ/n\BZ$. Then $G$ acts on $\bC\times \bX$
by multiplying the coordinate on $\bX$ with the $n-$th roots of unity. \\

A parabolic sheaf $\CF_\bullet$ is completely determined by the flag of sheaves
$$
\CF_0(-\bD_0)\subset \CF_{-n+1}\subset...\subset \CF_0,
$$
satisfying conditions~\ref{PS}.(a--e).
To $\CF_\bullet$ we can associate a single
$G$-invariant sheaf $\tilde \CF$ on $\bC\times \bX$:
$$
\tilde \CF=\sigma^*\CF_{-n+1}+\sigma^*\CF_{-n+2}(-\bD_0)+...+\sigma^*\CF_0(-(n-1)\bD_0).
$$
This $G$ invariant sheaf determines a parabolic sheaf if and only if it satisfies certain numeric and framing conditions that
mimick conditions~\ref{PS}.(b--d):

(b$'$) $ch_1(\tilde \CF)=-\frac {n(n-1)}2[\bD_0]$;

(c$'$) $ch_2(\tilde \CF)=d_1+...+d_n$;

(d$'$) $\tilde{\CF}$ is locally free at $\bD_\infty$ and trivialized there:

$$
\tilde \CF|_{\bD_\infty}=\CO_{\bD_\infty} \oplus \CO_{\bD_\infty}(-\bD_0) \oplus... \oplus \CO_{\bD_\infty}(-(n-1)\bD_0).
$$

If $\CF_\bullet$ is a $\widetilde{T}\times\BC^*\times\BC^*$ fixed
parabolic sheaf corresponding to a collection $\blambda$ as in the
previous section, then we have
\begin{equation} \label{explicit tilde fixed points}
\tilde \CF = \bigoplus_{l=1}^n J_{\lambda^l}(-(l-1)\bD_0)w_l,
\end{equation}
where $(\lambda^1,...,\lambda^n)$ is a collection of partitions, given by
\begin{equation} \label{relation lambdas}
\lambda^l_{ni-n\lfloor \frac {k-l}n \rfloor +k-l}=\lambda^{kl}_i.
\end{equation}
Here $\lfloor \frac {k-l}n \rfloor$ stands for the maximal integer smaller
than or equal to $\frac {k-l}n$.

For $j\in \BZ$, let $(j \textrm{ mod }n)$ denote that element of
$\{1,\ldots,n\}$ which is congruent to $j$ modulo $n$. For
$i\geq j\in \BZ$, if we denote
\begin{equation} \label{relation d lambda}
d_{ij}=\lambda^{j \textrm{ mod } n}_{i-j}
\end{equation}
we obtain a collection $(d_{ij})=\widetilde{\ul{d}}=\widetilde{\ul{d}}(\blambda)$ of non-negative integers with the properties that
\begin{equation}
\label{fei}
d_{kj}\geq d_{ij}\ \forall i\geq k\geq j;\
d_{i+n,j+n}=d_{ij}\ \forall i\geq j;\
d_{ij}=0\ \operatorname{for}\ i-j\gg0.
\end{equation}
For $1\leq k\leq n$, let us write
$$
d_k(\widetilde{\ul{d}})=\sum_{j\leq k}d_{kj}=\sum_{l=1}^n \sum_{i\leq \lfloor \frac {k-l}n\rfloor } d_{k(l+ni)}=$$
$$=\sum_{l=1}^n \sum_{i\geq 0} \lambda_{ni-n\lfloor \frac {k-l}n\rfloor +k-l}^l=\sum_{l=1}^{n}\sum_{i\geq 0} \lambda^{kl}_i=d_k(\blambda).
$$
The collection $(d_1(\widetilde{\ul{d}}),\ldots,d_n(\widetilde{\ul{d}}))=
(d_1(\blambda),\ldots,d_n(\blambda))$ will be denoted by 
$\|\widetilde{\ul{d}}\|=\|\blambda\|$.
Summarizing the above discussion, we have:
\begin{lem}
\label{dent}
The correspondence $\blambda\mapsto\widetilde{\ul{d}}(\blambda)$ is a
bijection between the set of collections $\blambda$ satisfying~(\ref{pool}),
and the set $D$ of collections $\widetilde{\ul{d}}$ satisfying~(\ref{fei}).
We have $\|\blambda\|=\|\widetilde{\ul{d}}(\blambda)\|$.
\end{lem}

By virtue of Lemmas~\ref{evi} and~\ref{dent} we will parametrize and
sometimes denote the $\widetilde{T}\times\BC^*\times\BC^*$-fixed points in
$\CP_{\ul{d}}$ by collections $\widetilde{\ul{d}}$ such that
$\ul{d}=\|\widetilde{\ul{d}}\|$.

\subsection{Correspondences}
If the collections $\ul{d}$ and $\ul{d}'$ differ at the only place
$i\in I:=\BZ/n\BZ$, and $d'_i=d_i+1$, then we consider the
correspondence
$\sE_{\ul{d},i}\subset\CP_{\ul{d}}\times\CP_{\ul{d}'}$ formed by
the pairs $(\CF_\bullet,\CF'_\bullet)$ such that for $j\not\equiv
i\pmod{n}$ we have $\CF_j=\CF'_j$, and for $j\equiv i\pmod{n}$ we
have $\CF'_j\subset\CF_j$.

It is a smooth quasiprojective algebraic variety of dimension
$1+2\sum_{i\in I}d_i$.

We denote by $\bp$ (resp. $\bq$) the natural projection
$\sE_{\ul{d},i}\to\CP_{\ul{d}}$ (resp.
$\sE_{\ul{d},i}\to\CP_{\ul{d}'}$). For $j\equiv i\pmod{n}$ the
correspondence $\sE_{\ul{d},i}$ is equipped with a natural line
bundle $\sL_j$ whose fiber at $(\CF_\bullet,\CF'_\bullet)$ equals
$\Gamma(\bC,\CF_j/\CF'_j)$. Finally, we have a transposed
correspondence
$^\sT\sE_{\ul{d},i}\subset\CP_{\ul{d}'}\times\CP_{\ul{d}}$.

\subsection{Equivariant cohomology}
We denote by ${}'M$ the direct sum of equivariant (complexified)
cohomology:
${}'M=\oplus_{\ul{d}}H^\bullet_{\widetilde{T}\times\BC^*\times\BC^*}
(\CP_{\ul{d}})$.
It is a module over $H^\bullet_{\widetilde{T}\times\BC^*\times\BC^*}(pt)
=\BC[\ft\oplus\BC\oplus\BC]=\BC[x_1,\ldots,x_n,\hbar,\hbar']$.
Here $\hbar'$ is twice the positive generator of $H^2_{\BC^*}(pt,\BZ)$ for
the second copy of $\BC^*$.
We define
$M=\ {}'M\otimes_{H^\bullet_{\widetilde{T}\times\BC^*\times\BC^*}(pt)}
\on{Frac}(H^\bullet_{\widetilde{T}\times\BC^*\times\BC^*}(pt))$.

We have an evident grading $$M=\oplus_{\ul{d}} M_{\ul{d}},\
M_{\ul{d}}=H^\bullet_{\widetilde{T}\times\BC^*\times\BC^*}(\CP_{\ul{d}})
\otimes_{H^\bullet_{\widetilde{T}\times\BC^*\times\BC^*}(pt)}
\on{Frac}(H^\bullet_{\widetilde{T}\times\BC^*\times\BC^*}(pt)).$$


\subsection{The action of generators}
\label{aff op} The grading and the correspondences
$^\sT\sE_{\ul{d},i},\sE_{\ul{d},i}$
give rise to the following operators on $M$
(note that though $\bp$ is not proper, $\bp_*$ is well defined on the
localized equivariant cohomology due to the finiteness of the fixed point
set of $\widetilde{T}\times\BC^*\times\BC^*$):
\begin{equation}
\label{ein}
\fh_i=\hbar^{-1}(x_{i+1}-x_i)+\delta_{i,0}\hbar^{-1}\hbar'+
2d_i-d_{i-1}-d_{i+1}+1:\ M_{\ul{d}}\to M_{\ul{d}};
\end{equation}
\begin{equation}
\label{drei}
\ff_i=\bp_*\bq^*:\ M_{\ul{d}}\to M_{\ul{d}-i};
\end{equation}
\begin{equation}
\label{vier}
\fe_i=-\bq_*\bp^*:\ M_{\ul{d}}\to M_{\ul{d}+i}.
\end{equation}

\begin{thm}
\label{braval}
For $n>2$,
the operators $\fe_i,\fh_i,\ff_i$ of~\ref{aff op} on $M$
satisfy the relations of the Chevalley generators of the Kac-Moody algebra
$\widehat{\mathfrak{sl}}_n$.
\end{thm}

The proof is entirely similar to the proof of Conjecture~3.7 of~\cite{bf}
in~3.8--3.10 of {\em loc. cit.} \qed

\begin{rem}
\label{refe}
{\em The paper~\cite{bf} studies the equivariant $K$-theory of $\CP_{\ul{d}}$,
so the computations involved in the proof of~Theorem~\ref{braval} are the
rational (simpler) versions of the trigonometric computations of~\cite{bf}.
We believe~Theorem~\ref{braval} holds true for $n=2$ (as well as~Conjecture~3.7
of~\cite{bf}). However, our method of verification of the relations between the
Chevalley generators consists in reduction to the finite case $d_0=0$ considered
in~Section~\ref{odin}. Since the Cartan matrix of $\widehat\fsl_2$ contains the 
off-diagonal entries $-2$, contrary to the Cartan matrix of $\fsl_n$, this
method fails for $n=2$. Alternatively,~Theorem~\ref{braval} can be proved by
direct calculations using the explicit matrix coefficients of the Chevalley
generators given in~Theorem~\ref{affeigin} below. These calculations, again,
are the rational versions of the trigonometric calculations of~\cite{ts}.}
\end{rem}

Similarly to Theorem~\ref{feigin}, it is possible to compute the matrix
coefficients of $\fe_i,\ff_i$ in the basis formed by the fundamental classes
of the $\widetilde{T}\times\BC^*\times\BC^*$-fixed points
$[\widetilde{\ul{d}}]\in M$. To this end let us assign to a collection
$\widetilde{\ul{d}}$ a collection of weights
$p_{ij}:=-x_{j\pmod{n}}+d_{ij}\hbar+\lfloor\frac{-j}{n}\rfloor\hbar'$.

\begin{thm}
\label{affeigin}
The matrix coefficients of the operators $\fe_i,\ff_i$ in the basis
$\{[\widetilde{\ul{d}}]\}$ are as follows:
$$\fe_{i[\widetilde{\ul{d}},\widetilde{\ul{d}}{}']}=
-\hbar^{-1}\frac{p_{i-1,j}-p_{ij}}{p_{ii}-p_{ij}}\prod_{j\ne k\leq i-1}
\frac{p_{i-1,k}-p_{ij}}{p_{ik}-p_{ij}}$$ if
$d'_{i,j}=d_{i,j}+1$ for certain $j\leq i$ (note that almost all factors in
this product are equal to 1 due to the condition~(\ref{fei}); also, in case
$i=j$ the factor $\frac{p_{i-1,j}-p_{ij}}{p_{ii}-p_{ij}}$ is set to be 1);

$$\ff_{i[\widetilde{\ul{d}},\widetilde{\ul{d}}{}']}=
\hbar^{-1}(p_{i+1,j}-p_{ij})(p_{i+1,i+1}-p_{ij})\prod_{j\ne k\leq i}
\frac{p_{i+1,k}-p_{ij}}{p_{ik}-p_{ij}}$$ if
$d'_{i,j}=d_{i,j}-1$ for certain $j\leq i$;

All the other matrix coefficients of $\fe_i,\ff_i$ vanish.
\end{thm}

The proof will be given in section~\ref{310} below.

\subsection{Universal Verma module over $\widehat{\fgl}_n$}
\label{hypo}
Contrary to the case of Laumon spaces (Theorem~\ref{brav}), $M$ is not
isomorphic to the universal Verma module over $\widehat{\fsl}_n$.
We conjecture that $M$ is isomorphic to the universal Verma module over
$\widehat{\fgl}_n$. Let us introduce some more notations and
correspondences in order to formulate the conjecture more precisely.

Note that the centre of $\widehat{\fsl}_n$ is spanned by the element
$C=\fh_0+\ldots+\fh_{n-1}$ which acts on $M$ as $n+\hbar'\hbar^{-1}$.
We denote by $\fH$ the Heisenberg Lie algebra with generators $a_i,\ i\in\BZ$,
and $C'$, and relations $[a_p,C']=0,\ [a_p,a_q]=\delta_{p,-q}pC'$.
We denote by $\widehat{\fgl}_n$ the quotient
$(\fH\oplus\widehat{\fsl}_n)/(C'-nC)$. The universal Verma module over
$\widehat{\fgl}_n$ is the tensor product of the universal Verma module over
$\widehat{\fsl}_n$ and the Fock module over $\fH$.

For $\ul{d}=(d_1,\ldots,d_n)$ and $m\geq1$
we set $\ul{d}+m\delta:=(d_1+m,\ldots,d_n+m)$.
Let $\fE_{\ul{d},m\delta}\subset\CP_{\ul{d}}\times\CP_{\ul{d}+m\delta}$ be
the correspondence formed by all the pairs $(\CF_\bullet,\CF'_\bullet)$
such that $\CF_\bullet\supset\CF'_\bullet$, and the quotient
$\CF_\bullet/\CF'_\bullet$ is supported at a single point in $\bC$.
In this case the collection of quotients $(\CF_1/\CF'_1,\ldots,\CF_n/\CF'_n)$
can be organized into a nilpotent representation $T_\bullet$ of the cyclic
quiver $\widetilde{A}_{n-1}$, see~\cite{fgk}~7.4.
Let $\sE_{\ul{d},m\delta}^\circ\subset\fE_{\ul{d},m\delta}$
be the locally closed
subset formed by all the pairs $\CF_\bullet\supset\CF_\bullet$ such that
the corresponding $\widetilde{A}_{n-1}$-representation is indecomposable.
According to Proposition~7.8 of~\cite{fgk}, $\sE_{\ul{d},m\delta}^\circ$
is a union of $n$ middle-dimensional irreducible components. We denote by
$\sE_{\ul{d},m\delta}$ the closure of $\sE_{\ul{d},m\delta}^\circ$.

The correspondence
$\sE_{\ul{d},m\delta}\subset\CP_{\ul{d}}\times\CP_{\ul{d}+m\delta}$
gives rise to the operator $M_{\ul{d}}\to M_{\ul{d}+m\delta}$ which we
denote by $a_m$. The transpose correspondence gives rise to the operator
$M_{\ul{d}+m\delta}\to M_{\ul{d}}$ which we denote by $a_{-m}$. We set
$a_0=n+\frac{\hbar'}{\hbar}$. The following conjecture was proposed by
A.~Kuznetsov a few years ago.

\begin{conj}
\label{kuzn}
a) The operators $a_m,\ m\in\BZ$ and the operators
$\fe_i,\fh_i,\ff_i$ of~(\ref{ein},\ref{drei},\ref{vier}) satisfy the relations
of the Lie algebra $\widehat{\fgl}_n$, and equip $M$ with a structure
of $\widehat{\fgl}_n$-module.

b) $M$ is isomorphic to the universal Verma module over $\widehat{\fgl}_n$.
\end{conj}

\begin{rem}
\label{baran}
{\em V.~Baranovsky~\cite{ba} has defined the action of Heisenberg algebra on 
the cohomology of Giesecker moduli spaces (of torsion free sheaves on surfaces)
via certain correspondences $\sB_{d,m}$. More precisely, let $\fM_{n,d}$ be 
the moduli space of torsion free sheaves
on $\bC\times\bX$ of rank $n$ and second Chern class $d$, trivialized at
$\bC\times\infty_\bX\cup\infty_\bC\times\bX$ (see~\cite{nak},~section~2).
We have an evident morphism $\eta:\ \CP_{\ul{d}}\to\fM_{n,d_0}$ 
(forgetting the flag). The correspondence $\sB_{d,m}\subset\fM_{n,d}\times
\fM_{n,d+m}$ is formed by the pairs $(\CF,\CF')$ such that $\CF\supset\CF'$,
and the quotient $\CF/\CF'$ is supported at a single point $x$ of $\BA^2$.
We consider the open piece $\overset{\circ}\sB_{d,m}$ given by the condition
that $x$ lies out of $\bC\times0_\bX$. The inverse image of 
$\overset{\circ}\sB_{d_0,m}$ under $\eta$ gives a correspondence 
$\overset{\circ}\sB_{\ul{d},m\delta}\subset\CP_{\ul{d}}\times\CP_{\ul{d}+m\delta}$.
It seems likely that as $x$ tends to $\bC\times0_\bX$, the correspondence
$\overset{\circ}\sB_{\ul{d},m\delta}$ tends to $\sE_{\ul{d},m\delta}$.}
\end{rem}

\subsection{Affine Yangian}
\label{guay}
From now on we impose the restriction $n>2$.
Let $(a_{kl})_{1\leq k,l\leq n}=\widehat{A}_{n-1}$ stand for the
Cartan matrix of
$\widehat{\fsl}_n$. The Yangian $Y(\widehat{\fsl}_n)$ is the free
$\BC[\hbar]$-algebra generated
by $\bx_{k,r}^\pm,\bh_{k,r},\ 1\leq k\leq n,\ r\in\BN$, with the
relations~(\ref{11}--\ref{15}) where $k,l$ are understood as residues
modulo $n$, so that for instance if $k=n$ then $k+1=1$.

The affine Yangian $\widehat Y$ of type $\widehat{A}_{n-1}$ is the free
$\BC[\hbar,\hbar']$-algebra generated by
$\bx_{k,r}^\pm,\bh_{k,r},\ 1\leq k\leq n,\ r\in\BN$ with the same relations
as in $Y(\widehat{\fsl}_n)$ except for relations~(\ref{12},\ref{14}) for the
pairs $(k,l)=(n,1),(1,n)$. These relations are modified as follows.
We introduce the shifted generating series
$\bh_n(u-\frac{\hbar'}{\hbar}-\frac{n}{2})=:{}'\bh_n(u):=
1+\sum_{r=0}^\infty{}'\bh_{n,r}\hbar^{-r}u^{-r-1};\
\bx_n^\pm(u-\frac{\hbar'}{\hbar}-\frac{n}{2})=:{}'\bx_n^\pm(u):=
\sum_{r=0}^\infty{}'\bx_{k,r}^\pm\hbar^{-r}u^{-r-1}$.
Now the new relations read
\begin{equation}
\label{122}
2['\bh_{n,r+1},\bx_{1,s}^\pm]-2['\bh_{n,r},\bx_{1,s+1}^\pm]=
\mp\hbar('\bh_{n,r}\bx_{1,s}^\pm+\bx_{1,s}^\pm{}{}'\bh_{n,r}),
\end{equation}
\begin{equation}
\label{123}
2[\bh_{1,r+1},{}'\bx_{n,s}^\pm]-2[\bh_{1,r},{}'\bx_{n,s+1}^\pm]=
\mp\hbar(\bh_{1,r}{}'\bx_{n,s}^\pm+{}'\bx_{n,s}^\pm{}\bh_{1,r}),
\end{equation}
\begin{equation}
\label{144}
2['\bx_{n,r+1}^\pm,\bx_{1,s}^\pm]-2['\bx_{n,r}^\pm,\bx_{1,s+1}^\pm]=
\mp\hbar('\bx_{n,r}^\pm\bx_{1,s}^\pm+\bx_{1,s}^\pm{}'\bx_{n,r}^\pm).
\end{equation}

Thus we have $Y(\widehat{\fsl}_n)=\widehat{Y}/(\hbar'+\frac{n\hbar}{2})$.

\begin{rem}
\label{ref}
{\em It is possible to define the Yangian $Y(\widehat\fsl_2)$ modifying the 
Serre relations~(\ref{15}). However, the above definition of the affine Yangian
$\widehat Y$ of type $\widehat{A}_{n-1}$ makes no sense for $n=2$. We do not
know any reasonable definition of $\widehat Y$ for $n=2$.}
\end{rem}

Note that $\widehat{Y}$ is isomorphic to $\widehat{\mathbf Y}_{\beta,\lambda}$
introduced in~\cite{g}~Definition~3.3. Indeed, the Yangian $\widehat{\mathbf Y}_{\beta,\lambda}$ is generated by $X_{k,r}^{\pm}$ and $H_{k,r}^{\pm}$, satisfying the relations~(\ref{11},\ref{13},\ref{15}) and some modification of
the relations~(\ref{12},\ref{14}). These modified relations can be rewritten in terms of the generating series $X_k^{\pm}(u)=\sum\limits_{r=0}^{\infty}X_{k,r}^{\pm}\lambda^{-r}u^{-1-r}$, $H_k(u)=\sum\limits_{r=0}^{\infty}H_{k,r}^{\pm}\lambda^{-r}u^{-1-r}$. The series $X_k^{\pm}(u)$, $H_l(u)$ satisfy the relations~(\ref{12'},\ref{14'}) except that for the
pairs $(k,l)=(1,n),(n,n-1)$ the relations are modified in the following way:

\begin{equation}
\label{12''}
\partial_u\partial_vH_k(u)X_l^\pm(v)(2u-2v\pm 1-\frac{2\beta}{\lambda}+1)=-\partial_u\partial_vX_l^\pm(v)H_k(u)(2v-2u\pm 1+\frac{2\beta}{\lambda}-1)
\end{equation}

\begin{equation}
\label{12'''}
\partial_u\partial_vH_l(u)X_k^\pm(v)(2u-2v\pm 1+\frac{2\beta}{\lambda}-1)=-\partial_u\partial_vX_k^\pm(v)H_l(u)(2v-2u\pm 1-\frac{2\beta}{\lambda}+1)
\end{equation}

\begin{equation}
\label{14''}
\partial_u\partial_vX_k^\pm(u)X_l^\pm(v)(2u-2v\pm 1-\frac{2\beta}{\lambda}+1)=-\partial_u\partial_vX_l^\pm(v)X_k^\pm(u)(2v-2u\pm 1+\frac{2\beta}{\lambda}-1)
\end{equation}

The isomorphism $\widehat{Y}\to\widehat{\mathbf Y}_{\beta,\lambda}$
takes $\hbar$ to $\lambda$, and $\hbar'$ to $-\frac{n\lambda}{2}+2\beta-\lambda$. It is
defined on the generating series as
\begin{equation}
\label{iso111} \bx_k^\pm(u)\mapsto X_k^\pm(u),\ \bh_k(u)\mapsto
H_k(u)\ \text{for}\ k\ne n,
\end{equation}
\begin{equation}
\label{iso222} \bx_n^\pm(u-\frac{\hbar'}{2\hbar}-\frac{n}{4})\mapsto X_n^\pm(u),\
\bh_n(u-\frac{\hbar'}{2\hbar}-\frac{n}{4})\mapsto H_n(u).
\end{equation}

\subsection{The action of the affine Yangian generators}
\label{quotients}
For any $m\leq i\in\BZ$ we will denote by $\ul{\CW}{}_{mi}$ the
quotient $\ul{\CF}{}_i/\ul{\CF}{}_m$ of the tautological
vector bundles, living on $\CP_{\ul{d}}\times\bC\subset\CP_{\ul{d}}\times\bS$.
Once again, $\pi:\CP_{\ul{d}}\times(\bC \ \backslash \infty) \rightarrow
\CP_{\ul{d}}$ denotes the standard projection. Let us consider:

\begin{equation}
\label{eqn:def}
\ba_{mi}(u)=u^{i-m}\cdot\bp_*(c(\pi_*(\ul{\CW}_{mi}|_{\bC \backslash \infty}), 
(-u\hbar)^{-1})\cdot\bq^*):\ M_{\ul{d}}\to M_{\ul{d}}[[u^{-1}]][u]
\end{equation}

\begin{cor}
\label{faktory}
The expression
$\bh_i(u):=\ba_{mi}(u+\frac{i-1}{2})^{-1}\ba_{mi}
(u+\frac{i+1}{2})^{-1}\ba_{m,i-1}(u+\frac{i-1}{2})
\ba_{m,i+1}(u+\frac{i+1}{2}):\ M_{\ul{d}}\to M_{\ul{d}}[[u^{-1}]]$
is independent of $m<i$.
\end{cor}

The proof follows from Corollary~\ref{faktor} by reduction to the stack
$\fZ_N$, cf.~\cite[Section~3.11]{bf}, and the proof of~Corollary~\ref{faktor}. 
Alternatively, it can be proved by a direct computation which is a rational
version of the trigonometric computation in the proof 
of~\cite[Corollary~4.12]{ts}.
\qed

\bigskip

Recall that $q$ stands for the character of
$\widetilde{T}\times\BC^*\times\BC^*:\ (\ul{t},v,c)\mapsto v^2$.
We define the line bundle $\sL'_k:=q^{\frac{1-k}{2}}\sL_k$ on the
correspondence $\sE_{\ul{d},k}$, that is $\sL'_k$ and $\sL_k$ are
isomorphic as line bundles
but the equivariant structure of $\sL'_k$ is obtained
from the equivariant structure of $\sL_k$ by the twist by the
character $q^{\frac{1-k}{2}}$.

For $1\leq k\leq n$ we consider the following
generating series of operators on $M$:

\begin{equation}
\label{razz}
\bh_k(u)=:1+\sum_{r=0}^\infty\bh_{k,r}\hbar^{-r}u^{-r-1}:\
M_{\ul{d}}\to M_{\ul{d}};
\end{equation}

\begin{equation}
\label{dvasstriss}
\bx_{k}^\pm(u)=:\sum_{r=0}^\infty\bx^\pm_{k,r}\hbar^{-r}u^{-r-1}:\
M_{\ul{d}}\to M_{\ul{d}\mp k}[[u^{-1}]],
\end{equation}

where

\begin{equation}
\label{dvass} \bx_{k,r}^+:=\bp_*(c_1(\sL'_k)^r\cdot\bq^*):\
M_{\ul{d}}\to M_{\ul{d}-k};
\end{equation}

\begin{equation}
\label{triss} \bx_{k,r}^-:=-\bq_*(c_1(\sL'_k)^r\cdot\bp^*):\
M_{\ul{d}}\to M_{\ul{d}+k}.
\end{equation}

\begin{thm}
\label{vara}
For $n>2$ the operators $\bh_{k,r},\bx_{k,r}^\pm$ on $M$ defined
in~(\ref{razz},\ref{dvass},\ref{triss})
satisfy the relations in $\widehat{Y}$, i.e.
they give rise to the action of $\widehat{Y}$ on $M$.
\end{thm}

\proof
For arbitrary $k\in\BZ$ we define $\bx_{k,r}^\pm:\ M_{\ul{d}}\to
M_{\ul{d}\mp k}$ by the same formulas~(\ref{dvass},\ref{triss}).

Let $q'$ stand for the character of $\widetilde{T}\times\BC^*\times\BC^*:\
(\ul{t},v,c)\mapsto c^2$. Then we have
$\ul{\CW}{}_{m-n,i-n}=q'\ul{\CW}{}_{mi}$,
that is $\ul{\CW}{}_{m-n,i-n}$ and $\ul{\CW}{}_{mi}$ are isomorphic as vector
bundles but the equivariant structure of $\ul{\CW}{}_{m-n,i-n}$ is obtained
from the equivariant structure of $\ul{\CW}{}_{mi}$ by the twist by the
character $q'$. It follows that
$\ba_{m-n,i-n}(u)=\ba_{mi}(u-\frac{\hbar'}{\hbar})$, and hence
$\bh_{k-n}(u)=\bh_k(u-\frac{\hbar'}{\hbar}-\frac{n}{2})$.

Similarly, for the equivariant line bundles on the correspondence
$\sE_{\ul{d},k}$ we have $\sL_{k-n}=q\sL_k$, and hence
$\bx_{k-n}^\pm(u)=\bx_k^\pm(u-\frac{\hbar'}{\hbar}-\frac{n}{2})$. In particular,
$\bx_0^\pm(u)=\bx_n^\pm(u-\frac{\hbar'}{\hbar}-\frac{n}{2})={}'\bx_n(u)$.

Now the relations~(\ref{122},\ref{123},\ref{144}) follow again from
Theorem~\ref{var} by reduction to the stack
$\fZ_N$, cf. section~3.11 of~\cite{bf}.  Alternatively, the desired relations
can be proved by direct computations which are rational versions of the 
trigonometric computations in the proof of~\cite[Theorem~4.13]{ts}.
\qed

Recall the weights $p_{ij}:=-x_{j\pmod{n}}+d_{ij}\hbar+
\lfloor\frac{-j}{n}\rfloor\hbar'$
introduced before Theorem~\ref{affeigin}.

\begin{thm}
\label{affei}
The matrix coefficients of the operators $\bx_{i,r}^\pm$ in the
fixed point basis $\{[\widetilde{\ul{d}}]\}$ of $M$ are as follows:
$$\bx^-_{i,r[\widetilde{\ul{d}},\widetilde{\ul{d}}{}']}=
(p_{ij}+\frac{1-i}{2}\hbar)^r\fe_{i[\widetilde{\ul{d}},\widetilde{\ul{d}}{}']}=$$
$$-\hbar^{-1}(p_{ij}+\frac{1-i}{2}\hbar)^r
\frac{p_{i-1,j}-p_{ij}}{p_{ii}-p_{ij}}\prod_{j\ne k\leq i-1}
\frac{p_{i-1,k}-p_{ij}}{p_{ik}-p_{ij}}$$ if
$d'_{i,j}=d_{i,j}+1$ for certain $j\leq i$ (note that almost all factors in
this product are equal to 1 due to the condition~(\ref{fei}); also, in case
$i=j$ the factor $\frac{p_{i-1,j}-p_{ij}}{p_{ii}-p_{ij}}$ is set to be 1);

$$\bx^+_{i,r[\widetilde{\ul{d}},\widetilde{\ul{d}}{}']}=
(p_{ij}-\frac{1+i}{2}\hbar)^r\ff_{i[\widetilde{\ul{d}},\widetilde{\ul{d}}{}']}=$$
$$\hbar^{-1}(p_{ij}-\frac{1+i}{2}\hbar)^r(p_{i+1,j}-p_{ij})(p_{i+1,i+1}-p_{ij})
\prod_{j\ne k\leq i}\frac{p_{i+1,k}-p_{ij}}{p_{ik}-p_{ij}}$$ if
$d'_{i,j}=d_{i,j}-1$ for certain $j\leq i$;

All the other matrix coefficients of $\bx_{i,r}^\pm$ vanish.

The eigenvalue of $\bh_i(u)$ on $[\widetilde{\ul{d}}]$ equals
$$
\prod_{j\le i}(u+\frac{i+1}{2}-p_{ij})^{-1}(u+\frac{i-1}{2}-p_{ij})^{-1}
(u+\frac{i+1}{2}-p_{i+1,j+1})(u+\frac{i-1}{2}-p_{i-1,j-1})
$$
\end{thm}

\proof Follows immediately from Theorem~\ref{affeigin} and the definition
of $\bx_{i,r}^\pm$. The formula for $\bh_i(u)$ follows from the fact that
the eigenvalue of $\ba_{mi}(u)$ on $[\widetilde{\ul{d}}]$ is
$\prod_{j\le i}(u-p_{ij})\prod_{k\le m}(u-p_{mk})^{-1}$. \qed

\begin{thm}
\label{uglov}
$M$ is an irreducible $\widehat{Y}$-module.
\end{thm}

\proof
According to Theorem~\ref{affei}, the Gelfand-Tsetlin subalgebra of
$\widehat{Y}$ generated by $\bh_{i,r}$ acts diagonally in the basis
$\{[\widetilde{\ul{d}}]\}$ with pairwise distinct joint eigenvalues.
Therefore it suffices to check the following two things:
\begin{enumerate}
\item for each $[\widetilde{\ul{d}}]$ there is an index $i$ such that
$\bx_{i,0}^-[\widetilde{\ul{d}}]\ne0$;
\item for each $[\widetilde{\ul{d}}]\ne[\widetilde{\ul{0}}]$ there is an
index $i$ such that $\bx_{i,0}^+[\widetilde{\ul{d}}]\ne0$.
\end{enumerate}
Both follow directly from Theorem~\ref{affei}.
\qed


\subsection{Specialization of Gelfand-Tsetlin base}
\label{integrable} We fix a positive integer $K$ (a level). We
consider an $n$-tuple $\mu=(\mu_{1-n},\ldots,\mu_0)\in\BZ^n$ such
that
$\mu_0+K\geq\mu_{1-n}\geq\mu_{2-n}\geq\ldots\geq\mu_{-1}\geq\mu_0$.
We view $\mu$ as a dominant (integrable) weight of $\hat{\fgl}_n$ of
level $K$. We extend $\mu$ to a nonincreasing sequence
$\widetilde{\mu}=(\widetilde{\mu}_i)_{i\in\BZ}$ setting
$\widetilde{\mu}_i:=\mu_{i\pmod{n}}+\lfloor\frac{-i}{n}\rfloor K$.

We define a subset $D(\mu)$ ({\em affine Gelfand-Tsetlin patterns})
of the set $D$ of all collections
$\widetilde{\ul{d}}$ satisfying the conditions~(\ref{fei}) as
follows:
\begin{equation}
\label{ryb} \widetilde{\ul{d}}\in D(\mu)\ \on{iff}\
d_{ij}-\widetilde{\mu}_j\leq d_{i+l,j+l}-\widetilde{\mu}_{j+l}\
\forall\ j\leq i,\ l\geq0.
\end{equation}
We specialize the values of $x_1,\ldots,x_n,\hbar,\hbar'$ so that
\begin{equation}
\label{rybn} \hbar=1,\ \hbar'=-K-n,\ x_j=\widetilde{\mu}_j-j+1.
\end{equation}
We define the renormalized vectors
\begin{equation}
\label{renorm}
\langle\widetilde{\ul{d}}\rangle:=
C^{-1}_{\widetilde{\ul{d}}}[\widetilde{\ul{d}}]
\end{equation}
where $C_{\widetilde{\ul{d}}}$ is the product
$\prod_{w\in T_{\widetilde{\ul{d}}}\CP_{\ul{d}}}w$ of the weights of
$\widetilde{T}\times\BC^*\times\BC^*$ in the tangent space to $\CP_{\ul{d}}$
at the point $\widetilde{\ul{d}}$. The explicit formula for the multiset
$\{w\}$ is given in Proposition~\ref{torus char E} below. Namely, one has
to plug $\widetilde{\ul{d}}{}'=\widetilde{\ul{d}}$ into the formula of
Proposition~\ref{torus char E}, then expand it as a sum of Laurent monomials in
$\ul{t},q,q'$ with positive integer coefficients, and then replace each
monomial $q^b(q')^c\prod_{i=1}^nt_i^{2a_i}$ with the corresponding weight
$w=b\hbar+c\hbar'+\sum_{i=1}^na_ix_i$ (with multiplicity given by the
corresponding positive integer coefficient).

We define $V(\mu)$
as the $\BC$-linear span of the vectors
$\langle\widetilde{\ul{d}}\rangle$ for $\widetilde{\ul{d}}\in
D(\mu)$.

\begin{thm}
\label{rybni} The formulas of Theorem~\ref{affei} give rise to the
action of $\widehat{Y}/(\hbar-1,\hbar'+K+n)$ in $V(\mu)$.
\end{thm}

\proof It follows from~(\ref{ratiu}) below
that for the matrix coefficients in the
renormalized basis $\{\langle\widetilde{\ul{d}}\rangle\}$ we have
$\fe_{i\langle\widetilde{\ul{d}},\widetilde{\ul{d}}{}'\rangle}=
-\ff_{i[\widetilde{\ul{d}}{}',\widetilde{\ul{d}}]},\
\ff_{i\langle\widetilde{\ul{d}},\widetilde{\ul{d}}{}'\rangle}=
-\fe_{i[\widetilde{\ul{d}}{}',\widetilde{\ul{d}}]}$. We have to
check two things:

a) for $\widetilde{\ul{d}}\in D(\mu)$ the denominators of the matrix
coefficients
$\bx^\pm_{i,r\langle\widetilde{\ul{d}},\widetilde{\ul{d}}{}'\rangle}$
do not vanish;

b) for $\widetilde{\ul{d}}\in D(\mu),\ \widetilde{\ul{d}}{}'\not\in
D(\mu)$ the numerators of the matrix coefficients
$\bx^\pm_{i,r\langle\widetilde{\ul{d}},\widetilde{\ul{d}}{}'\rangle}$
do vanish. Both are straightforward. \qed

\begin{rem}
\label{suzuki}
{\em It seems likely that applying the Schur-Weyl functor of~\cite{g} to
$V(\mu)$ and then going to the limit $n\to\infty$ one obtains the 
``tableaux representations'' (see~section~3 of~\cite{s})
of the trigonometric Cherednik algebra of type $A$.}
\end{rem}

Restricting $V(\mu)$ to $U(\hat{\fsl}_n)\subset\widehat Y$ we obtain
the same named $\hat{\fsl}_n$-module. Recall the embedding
$\hat{\fsl}_n\subset\hat{\fgl}_n$ of~section~\ref{hypo}.

\begin{conj}
\label{hypot} The $\hat{\fsl}_n$-module $V(\mu)$ extends to the
irreducible integrable $\hat{\fgl}_n$-module with highest weight
$\mu$.
\end{conj}

\medskip

Though we do not know how to define the action of Heisenberg algebra
on $V(\mu)$, we can prove the following weak version of
Conjecture~\ref{hypot}.

\begin{thm}
\label{okun} The character of $V(\mu)$ equals the character of the
irreducible integrable $\hat{\fgl}_n$-module with highest weight
$\mu$.
\end{thm}

\proof We will use a combinatorial model for integrable
$\widehat{\fsl}_n$-crystals introduced by P.~Tingley in~\cite{t},
namely, the cylindric plane partitions model of section 4 of {\em
loc. cit.} We will denote by $\fB_\mu$ the set of cylindric plane
partitions with boundary $\mu$ ({\em loc. cit.},~Definition~4.6).
P.~Tingley introduces a structure of $\widehat{\fsl}_n$-crystal on
$\fB_\mu$ (section~4.2 of {\em loc. cit.}), and proves (Theorem~4.16
of {\em loc. cit.}) that this crystal is isomorphic to the
$\widehat{\fsl}_n$-crystal of the irreducible integrable
$\hat{\fgl}_n$-module with highest weight $\mu$.

Thus it remains to construct a weight-preserving bijection between
$D(\mu)$ and $\fB_\mu$ (adding a fourth combinatorial model to
Tingley's list). Namely, given a cylindric plane partition
$\pi\in\fB_\mu$ we view it in the three dimensional representation
of Figure~13 of {\em loc. cit.} We shift the boxes in each row
$x=x_0,\ z=z_0$ in $y$-direction so that the boundary (the
infinitely high wall) of $\pi$ becomes the plane $y=0$. We will
denote the resulting configuration of boxes by $\pi_{\on{red}}$.
Note that it is not a plane partition anymore: the heights of piles
are not necessarily nonincreasing in $x$-direction, though the
heights of piles in $y$-direction are still nonincreasing. Still we
can reconstruct $\pi$ from $\pi_{\on{red}}$ shifting back by
$\widetilde{\mu}$ in $y$-direction.

We will denote by $\pi^*_{\on{red}}$ the three dimensional picture
of $\pi_{\on{red}}$ reflected in the plane $z=y$. We denote the
height of pile of $\pi^*_{\on{red}}$ at $(x,y)$ by $h(\pi)_{x,y}$.
Here $x$ and $y$ are integers, and $y$ is nonnegative. Finally, we
set $d(\pi)_{ij}:=h(\pi)_{-j,i-j}$. It is easy to see that the
condition of $\pi$ being a cylindric plane partition with boundary
$\mu$ is equivalent to the condition $d(\pi)\in D(\mu)$. Thus we
have constructed the desired bijection. This completes the proof of
the theorem. \qed

\section{Cohomology rings of affine Laumon spaces}

\subsection{Affine Gelfand-Tsetlin subalgebra}
\label{nuzhen}
We extend the scalars in the $\BC[\hbar,\hbar']$-algebra $\widehat Y$
to $\BC(\hbar,\hbar')[x_1,\ldots,x_n]$, and denote the resulting algebra by
$\widetilde Y$. We define the affine
Gelfand-Tsetlin subalgebra $\fA_{\on{aff}}\subset\widetilde{Y}$ as the
subalgebra generated by all the elements $\bh_{i,r}$.
In this section we construct a surjective homomorphism from the affine
Gelfand-Tsetlin subalgebra $\fA_{\on{aff}}\subset\widetilde{Y}$ to the
localized
equivariant cohomology ring
$^{\on{loc}}H^\bullet_{\widetilde{T}\times\BC^*\times\BC^*}(\CP_{\ul{d}}):=
H^\bullet_{\widetilde{T}\times\BC^*\times\BC^*}(\CP_{\ul{d}})
\otimes_{\BC[\hbar,\hbar']}\BC(\hbar,\hbar')$. Namely, let
$1_{\ul{d}}\in H^\bullet_{\widetilde{T}\times\BC^*\times\BC^*}(\CP_{\ul{d}})$
denote the unit element of the cohomology ring. For $a\in\fA_{\on{aff}}$ we
set $\psi(a):=a(1_{\ul{d}})$ (the action of Theorem~\ref{vara}).

\begin{thm}
\label{feiryb}
The homomorphism $\psi:\ \fA_{\on{aff}}\to{}
^{\on{loc}}H^\bullet_{\widetilde{T}\times\BC^*\times\BC^*}(\CP_{\ul{d}})$
is surjective.
\end{thm}

The proof occupies the rest of the section.

\subsection{Modified generators}
\label{h2a}
First we are going to express the coefficients of $\ba_{0i}(u)$ as
polynomials in $\bh_{k,r}$. Note that $\ba_{00}(u)=1$ and
\begin{equation}\label{ai+n}\ba_{0,i+n}(u)\ba_{0,n}(u)^{-1}=\ba_{n,i+n}(u)=
\ba_{0i}(u+\frac{\hbar'}{\hbar}),\end{equation}
and hence it is sufficient to express $a_{0i}(u)$ for $i=1,\ldots,n$. We have
\begin{equation}\ba_{0i}(u+\frac{i-1}{2})^{-1}\ba_{0i}(u+\frac{i+1}{2})^{-1}
\ba_{0,i-1}(u+\frac{i-1}{2})
\ba_{0,i+1}(u+\frac{i+1}{2})=\bh_i(u).\end{equation}
By induction we obtain
\begin{equation}\label{ai}\ba_{0i}(u)=\prod\limits_{j=0}^{i-1}\ba_{01}(u-j)
\prod\limits_{j=1}^{i-1}\prod\limits_{l=1}^{i-j}\bh_j(u-l-\frac{j-1}{2}).
\end{equation}
Thus it remains to express $\ba_{01}(u)$. From~(\ref{ai+n}) for $i=1$
and~(\ref{ai}) for $i=n,\ n+1$, we have
\begin{equation}\label{a01}\ba_{01}(u-n)\prod\limits_{j=1}^{n}
\bh_j(u-n+\frac{j-1}{2})=\ba_{01}(u+\frac{\hbar'}{\hbar}).\end{equation}
For $\frac{\hbar'}{\hbar}\ne -n$ this equation uniquely determines
$\ba_{01}(u)$.

\begin{rem}
{\em The ``critical'' value $\frac{\hbar'}{\hbar}=-n$ corresponds to the level $0$ Verma module, and, on the other hand, to the
specialization of the affine Yangian $\widehat{Y}$ to $\beta=-\frac{n\lambda}{4}+\frac{\lambda}{2}$. This corresponds to the degenerate trigonometric DAHA $\bH_{0,c}$ by the Schur-Weyl duality (see~\cite{g}). There is an additional relation on $\bh_i(u)$ for
$\frac{\hbar'}{\hbar}=-n$, namely, $\prod\limits_{j=1}^{n}
\bh_j(u-n+\frac{j-1}{2})=1$.}
\end{rem}

\subsection{K\"unneth components of characteristic classes}
\label{a2c}

Now we are going to express the K\"unneth components of characteristic
classes of the tautological bundles $\ul\CF{}_i$ on
$\CP_{\ul{d}}\times\bC\times\bX$
in terms of the coefficients of $\ba_{0k}(u)$ (for all $k$). We start by noting that our definition \eqref{eqn:def} is equivalent to the following: under the K\"unneth decomposition, one writes $c_j(\ul{\CW}{}_{mi})=:c_j^{(j)}(\ul{\CW}{}_{mi})\otimes1+
c_j^{(j-1)}(\ul{\CW}{}_{mi})\otimes\tau$
where $c_j^{(j)}(\ul{\CW}{}_{mi})\in
H^{2j}_{\widetilde{T}\times\BC^*}(\CP_{\ul{d}})$, and
$c_j^{(j-1)}(\ul{\CW}{}_{mi})\in
H^{2j-2}_{\widetilde{T}\times\BC^*}(\CP_{\ul{d}})$. By equivariant localization on
$\bC\ \backslash \ \infty \cong \bA^1$, it is not hard to prove:

\begin{equation}
\label{eqneqneqn}
\ba_{mi}(u)=u^{i-m}+\sum_{r=1}^\infty(-\hbar)^{-r}\left(c_r^{(r)}(\ul{\CW}{}_{mi})-\hbar c_r^{(r-1)}
(\ul{\CW}{}_{mi})\right)u^{i-m-r}.
\end{equation}

Let $\tau$ and $\tau'$ stand for the first Chern class of the line bundle $\CO(1)$ on
$\bC$ and $\bX$, respectively. Then we have $c_j(\ul\CF{}_i)=:c_j^{(j)}(\ul\CF{}_i)\otimes 1+
c_j^{(j-1)}(\ul\CF{}_i)\otimes\tau+c_j^{(j-1)'}(\ul\CF{}_i)\otimes\tau'+
c_j^{(j-2)}(\ul\CF{}_i)\otimes\tau\tau'$
for the K\"unneth components $c_j^{(j)}(\ul\CF{}_i),c_j^{(j-1)}(\ul\CF{}_i),
c_j^{(j-1)'}(\ul\CF{}_i),c_j^{(j-2)}(\ul\CF{}_i)$. In order to prove that
all the K\"unneth components just defined are expressible in terms of
the coefficients of all $\ba_{0k}(u)$, it suffices to do this for just one
tautological bundle $\ul\CF{}_0$.

For an equivariant vector bundle $\CF$ on $\CP_{\ul{d}}\times\bC\times\bX$
let $X^{\widetilde{\ul{d}},0,0}(\CF)$ denote the multiset of characters of
$\widetilde{T}\times\BC^*\times\BC^*$ at the fixed point
$(\widetilde{\ul{d}},0,0)\in\CP_{\ul{d}}\times\bC\times\bX$.
Similarly, we define the multisets $X^{\widetilde{\ul{d}},0,\infty}(\CF),
X^{\widetilde{\ul{d}},\infty,0}(\CF),
X^{\widetilde{\ul{d}},\infty,\infty}(\CF)$.
Let $e_j^{\widetilde{\ul{d}},0,0}(\CF)$ stand for the sum of products of
$j$ distinct elements of the multiset $X^{\widetilde{\ul{d}},0,0}(\CF)$.
Similarly, we define $e_j^{\widetilde{\ul{d}},0,\infty}(\CF),
e_j^{\widetilde{\ul{d}},\infty,0}(\CF),
e_j^{\widetilde{\ul{d}},\infty,\infty}(\CF)\in
\BC[\hbar,\hbar',x_1,\ldots,x_n]$. Let $e_j^{0,0}(\CF)$ be the diagonal
operator in the basis $\{[\widetilde{\ul{d}}]\}$ with eigenvalues
$e_j^{\widetilde{\ul{d}},0,0}(\CF)$. Similarly, we define
$e_j^{0,\infty}(\CF),e_j^{\infty,0}(\CF),e_j^{\infty,\infty}(\CF)$.

We have
\begin{multline}
\label{marazm}
e_j^{0,0}(\CF)=c_j^{(j)}(\CF)-\hbar c_j^{(j-1)}(\CF)-\hbar'c_j^{(j-1)'}(\CF)+
\hbar\hbar'c_j^{(j-2)}(\CF)\\
e_j^{0,\infty}(\CF)=c_j^{(j)}(\CF)-\hbar c_j^{(j-1)}(\CF)+
\hbar'c_j^{(j-1)'}(\CF)-\hbar\hbar'c_j^{(j-2)}(\CF)\\
e_j^{\infty,0}(\CF)=c_j^{(j)}(\CF)+\hbar c_j^{(j-1)}(\CF)-
\hbar'c_j^{(j-1)'}(\CF)-\hbar\hbar'c_j^{(j-2)}(\CF)\\
e_j^{\infty,\infty}(\CF)=c_j^{(j)}(\CF)+\hbar c_j^{(j-1)}(\CF)+
\hbar'c_j^{(j-1)'}(\CF)+\hbar\hbar'c_j^{(j-2)}(\CF)
\end{multline}

For an equivariant vector bundle $\CG$ on $\CP_{\ul{d}}\times\bC$
we spare the reader the
bulk of the similar self-explaining notation but just note that we have
\begin{equation}
\label{marasm}
e_j^0(\CG)=c_j^{(j)}(\CG)-\hbar c_j^{(j-1)}(\CG),\ \ \
e_j^\infty(\CG)=c_j^{(j)}(\CG)+\hbar c_j^{(j-1)}(\CG)
\end{equation}
whence
\begin{equation}
\label{alzheimer}
c_j^{(j)}(\CG)=\frac{1}{2}(e_j^\infty(\CG)+e_j^0(\CG)),\ \ \
c_j^{(j-1)}(\CG)=\frac{1}{2\hbar}(e_j^\infty(\CG)-e_j^0(\CG))
\end{equation}

If we have $\CG=\CF|_{\CP_{\ul{d}}\times\bC\times0}$ then
$e_j^{0,0}(\CF)=e_j^0(\CG)$, that is
\begin{equation}
\label{senile}
c_j^{(j)}(\CF)-\hbar c_j^{(j-1)}(\CF)-\hbar'c_j^{(j-1)'}(\CF)+
\hbar\hbar'c_j^{(j-2)}(\CF)=c_j^{(j)}(\CG)-\hbar c_j^{(j-1)}(\CG).
\end{equation}
The equality~(\ref{senile}) makes sense in the equivariant $K$-groups of
$\CP_{\ul{d}}\times\bC\times\bX$ and $\CP_{\ul{d}}\times\bC$, so it holds
true for any coherent sheaf $\CF$ on $\CP_{\ul{d}}\times\bC\times\bX$
and its (left derived) restriction $\CG$ to $\CP_{\ul{d}}\times\bC\times0$.
Let us take $\CF=\ul{\CF}{}_0$, so that $\CG=\ul{\CW}{}_{-n,0}$,
and then \eqref{eqneqneqn} shows that the RHS of~(\ref{senile}) is just a coefficient of the series
$\ba_{-n,0}(u)$.
Note that $e_j^{0,\infty}(\ul{\CF}{}_n),e_j^{\infty,0}(\ul{\CF}{}_n)$,
and $e_j^{\infty,\infty}(\ul{\CF}{}_0)$ are all equal to the $j$-th
elementary symmetric function of the variables $x_1,\ldots,x_n$.
Now it is easy to see that the system~(\ref{marazm}) can be solved
over the ring $\BC(\hbar,\hbar')[x_1,\ldots,x_n]$ to express the K\"unneth
components $c_j^{(j)}(\ul\CF{}_0),c_j^{(j-1)}(\ul\CF{}_0),
c_j^{(j-1)'}(\ul\CF{}_0),c_j^{(j-2)}(\ul\CF{}_0)$ in terms of the
coefficients of the series $\ba_{-n,0}(u)$ (or equivalently, in terms of the
coefficients of the series $\ba_{0n}(u)$).

\subsection{Okounkov's vector bundle $E$}
\label{andrei}
Now it remains to prove that
$^{\on{loc}}H^\bullet_{\widetilde{T}\times\BC^*\times\BC^*}(\CP_{\ul{d}})$
is generated
by the K\"unneth components of the Chern classes of the tautological bundles
$\ul{\CF}{}_i$. By Theorem~2.1 in \cite{es}, it is enough to show that the cohomology class of the diagonal $\Delta \subset \CP_{\ul{d}} \times \CP_{\ul{d}}$ is of the form
\begin{equation} \label{form of diagonal}
[\Delta]=\sum_i \alpha_i \boxtimes \beta_i,
\end{equation}
where $\alpha_i,\beta_i \in{}^{\on{loc}}H^\bullet_{\widehat{T}\times
\BC^*\times \BC^*}(\CP_{\ul{d}})$ are generated by the K\"unneth components of the Chern classes of the tautological bundles. The following sections are concerned with proving this fact: we will compute the class $[\Delta]$ and show that it is of the form \eqref{form of diagonal}. \\

Recall the setup of subsection~\ref{realization}, where we realized a
parabolic sheaf $\CF_\bullet$ as a single, $G=\BZ/n\BZ-$invariant sheaf
$\tilde \CF$. Following~\cite{eo}, consider the sheaf $E$ on
$\CP_{\ul{d}} \times \CP_{\ul{d}}$,
whose fiber above $(\CF_\bullet, \CF'_\bullet)$ is
$$
E|_{(\CF_\bullet, \CF'_\bullet)}=\textrm{Ext}^1_G(\tilde{\CF'}, \tilde{\CF}(-\bD_\infty))
$$
In the above, $\textrm{Ext}^1_G$ denotes the $G-$invariant part of the vector space in question. By the Kodaira-Spencer theorem, the restriction of $E$ to the diagonal $\Delta \subset \CP_{\ul{d}}\times \CP_{\ul{d}}$ coincides with the tangent bundle of $\Delta\cong \CP_{\ul{d}}$.

\begin{prop} \label{E bundle}
The sheaf $E$ is a vector bundle of rank $2(d_1+\dots+d_n)$ on $\CP_{\ul{d}} \times \CP_{\ul{d}}$. Note that the rank of $E$ equals half of the dimension of the base space.
\end{prop}

\begin{proof} Since our sheaves are on the surface $\bS$, all $\textrm{Ext}^i$ groups vanish for $i>2$. However, in our particular case we can say more. Because $\tilde{\CF}$ and $\tilde{\CF'}$ have the same trivialization at $\bD_\infty$, they have the same first Chern class and thus $\textrm{Hom}(\tilde{\CF'},\tilde{\CF}(-\bD_\infty))=0$. Moreover, Serre duality implies that $\textrm{Ext}^2(\tilde{\CF'},\tilde{\CF}(-\bD_\infty))\cong \textrm{Hom}(\tilde{\CF},\tilde{\CF'}(-\bD_\infty))$, which vanishes for the same reason. This suggests that instead of $\textrm{Ext}^1$ one should consider the following functor of sheaves on $\bS$:
$$
\chi(\CS', \CS):=\textrm{Hom}(\CS',\CS)-\textrm{Ext}^1(\CS',\CS)+\textrm{Ext}^2(\CS',\CS).
$$
Then the above discussion implies that
$$
E|_{(\CF_\bullet, \CF'_\bullet)}=\textrm{Ext}^1_G(\tilde{\CF'}, \tilde{\CF}(-\bD_\infty))=-\chi_G(\tilde{\CF'}, \tilde{\CF}(-\bD_\infty)),
$$
where $\chi_G$ denotes the $G-$invariant part of the vector space in question. Therefore, Corollary~7.9.9 of~\cite{mum} implies that the dimension of the fibers of $E$ is
constant, and therefore $E$ is a vector bundle. The fact that its rank equals $2(d_1+\dots+d_n)$ follows from
Remark~\ref{rank of E}.
\end{proof}

\subsection{A section of $E$}
\label{aku}
Now we will construct a regular section of $E$ that vanishes on
the diagonal $\Delta$. The following construction was proposed by
A.~Kuznetsov. For any $(\CF_\bullet, \CF'_\bullet)\in \CP_{\ul{d}}\times \CP_{\ul{d}}$, consider the short exact sequence
$$
0\rightarrow \tilde{\CF}(-\bD_\infty)\rightarrow \tilde{\CF}\rightarrow \tilde{\CF}|_{\bD_\infty}\rightarrow 0
$$
Apply the functor $\textrm{Hom}_G(\tilde{\CF'}, \cdot)$ to the above short exact sequence, and we will obtain the following long exact sequence
\begin{equation} \label{les for s}
\cdots \rightarrow \textrm{Hom}_G(\tilde{\CF'}, \tilde{\CF})\rightarrow \textrm{Hom}_G(\tilde{\CF'}, \tilde{\CF}|_{\bD_\infty}) \rightarrow \textrm{Ext}^1_G(\tilde{\CF'}, \tilde{\CF}(-\bD_\infty))\rightarrow \cdots
\end{equation}
Since our parabolic sheaves are framed at $\bD_\infty$, then there is a fixed isomorphism $\tilde{\CF}|_{\bD_\infty}\cong \tilde{\CF'}|_{\bD_\infty}$. Therefore, the morphism which restricts $\tilde{\CF'}$ to $\tilde{\CF'}|_{\bD_\infty}\cong \tilde{\CF}|_{\bD_\infty}$ is an element of the middle $\textrm{Hom}$ space in the above exact sequence. Push this element forward to $\textrm{Ext}^1_G(\tilde{\CF'}, \tilde{\CF}(-\bD_\infty))$, and by definition this will be the value of our section $s(\CF_\bullet,\CF'_\bullet)$.

\begin{prop} \label{zero locus}
The zero locus of $s$ is precisely the diagonal $\Delta\subset \CP_{\ul{d}}\times \CP_{\ul{d}}$.
\end{prop}

\begin{proof} The fact that $s(\CF_\bullet, \CF_\bullet)=0$ easily follows from the long exact sequence \eqref{les for s}. Indeed, in this case the element of the middle $\textrm{Hom}$ space which we push forward to obtain $s(\CF_\bullet, \CF_\bullet)$ simply comes from the identity element in the left $\textrm{Hom}$ space. Since the composition of two successive maps in a long exact sequence is zero, this implies that $s(\CF_\bullet, \CF_\bullet)=0$.

Conversely, suppose $s(\CF_\bullet,\CF'_\bullet)=0$. By the exactness of \eqref{les for s}, it follows that the natural restriction morphism in the middle $\textrm{Hom}$ space comes from the left $\textrm{Hom}$ space. This implies that there exists a morphism $\phi\in \textrm{Hom}_G(\tilde{\CF'}, \tilde{\CF})$ which restricts to the identity on $\bD_\infty$. Since the sheaves $\tilde{\CF'}$ and $\tilde{\CF}$ have the same framing at $\bD_\infty$, such a $\phi$ can exist only if $\tilde{\CF'} \subset\tilde{\CF} \Leftrightarrow \CF'_\bullet \subset \CF_\bullet$ (in which case $\phi$ can only be the natural inclusion). Since $\CF_\bullet$ and $\CF'_\bullet$ have the same first and second Chern classes, this implies that $\CF'_\bullet=\CF_\bullet$.
\end{proof}

The zero-locus $[\Delta]$ is irreducible (since it is isomorphic to $\CP_{\ul{d}}$), and by Proposition~\ref{E bundle} its codimension in $\CP_{\ul{d}}\times \CP_{\ul{d}}$ precisely equals the rank of the bundle $E$. Therefore we have
\begin{equation} \label{diagonal euler class}
e(E)=k\cdot [\Delta],
\end{equation}
where $k\in \BN$ is the order of vanishing of the section $s$ at its zero-locus.

\begin{rem} {\em We conjecture that the section $s$ is transversal, 
and therefore $k=1$. However, we will not need this technical result.}
\end{rem}

\subsection{The class of $E$} 
\label{torus computations}
In this section, we will compute the class of $E$ in the
$\widetilde{T}\times \BC^*\times \BC^*$-equivariant $K$-theory of
$\CP_{\ul{d}}\times\CP_{\ul{d}}$. In the following, we will abuse notation
and denote by $\chi(\CF'_{k-1-n}, \CF_{k-n}(-\bD_\infty))$ and
$\chi(\CF'_{k-n}, \CF_{k-n}(-\bD_\infty))$ the sheaves over
$\CP_{\ul{d}}\times \CP_{\ul{d}}$ whose fibers above
$(\CF_\bullet,\CF'_\bullet)$ are the $\chi$ spaces in question.
By the argument of Proposition~\ref{E bundle}, these sheaves are vector
bundles. Recall the character $q'$ of
$\widetilde{T}\times \BC^*\times \BC^*$.

\begin{prop} \label{K-theory class}
We have the following equality of $K$-theory classes:
\begin{equation} \label{E K-theory}
[E]=\frac {\dsp \sum_{k=1}^n [\chi(\CF'_{k-n}, \CF_{k-n}(-\bD_\infty))] - \dsp \sum_{k=1}^n [\chi(\CF'_{k-1-n}, \CF_{k-n}(-\bD_\infty))]}{q'-1}.
\end{equation}
\end{prop}

\begin{rem} {\em The equivariant $K$-theory of $\CP_{\ul{d}}\times
\CP_{\ul{d}}$ is an algebra over the ring of scalars
$K_0:=K_{\widetilde{T}\times \BC^*\times \BC^*}(pt)$. Relation
\eqref{E K-theory} only makes sense in the localized equivariant $K$-theory
algebra, i.e. localized with respect to $q-1\in K_0$.}
\end{rem}

\begin{proof} By the Thomason localization theorem, restriction to the fixed point locus produces the following isomorphism of vector spaces
$$
K_{\widetilde{T}\times \BC^*\times \BC^*}(\CP_{\ul{d}}\times \CP_{\ul{d}}) \bigotimes_{K_0} \textrm{Frac}(K_0) \iso$$
$$K_{\widetilde{T}\times \BC^*\times \BC^*}(\CP_{\ul{d}}^{\widetilde{T}\times \BC^*\times \BC^*}\times \CP_{\ul{d}}^{\widetilde{T}\times \BC^*\times \BC^*}) \bigotimes_{K_0} \textrm{Frac}(K_0).
$$
Therefore, to prove the equality \eqref{E K-theory} it is enough to show that equality holds when restricting to every torus fixed point of $\CP_{\ul{d}}\times \CP_{\ul{d}}$. But the fixed points are isolated, and thus at each fixed point will we just have to check an isomorphism of $\widetilde{T}\times \BC^*\times \BC^*-$representations. Two representations of a torus are isomorphic if and only if their characters are equal. Thus, we only need to compute the characters of $\widetilde{T}\times \BC^*\times \BC^*$ in the fixed fibers of the left hand side and right hand side of \eqref{E K-theory}, and show that they are equal. \\

Recall from sections~\ref{fp} and \ref{realization} that a fixed point $(\CF_\bullet,\CF'_\bullet)\in \CP_{\ul{d}}\times \CP_{\ul{d}}$ is given by two collections of partitions $(\blambda,\blambda')$, or alternatively by two vectors of positive integers $(\widetilde{\ul{d}},\widetilde{\ul{d}}')$. The relation between $\blambda$ and $\widetilde{\ul{d}}$ is given by relations \eqref{relation lambdas} and \eqref{relation d lambda}: if $1\leq k,l\leq n$, then
$$
\lambda^{kl}_s=d_{ns-n\lfloor \frac {k-l}n \rfloor +k,l}
$$
The computation of the characters we need is based on the following lemma:

\begin{lem} \label{torus char 0}
Let $\bD_0=\bC \times 0_{\bX}$ and $\bD_1=0_{\bC}\times \bX$. Let us consider two partitions $\lambda,\lambda'$ and the associated rank 1 torsion free sheaves $J_\lambda, J_{\lambda'}$ on $\bS$. Then the torus character of $\BC^*\times \BC^*$ in $\chi(J_{\lambda'}, J_\lambda(-\bD_\infty+\alpha \bD_0+\beta \bD_1))$ equals
$$
\emph{char}_{\alpha\beta}(\lambda',\lambda):=-\sum_{i=0}^\infty\sum_{i'=0}^\infty q^{\beta+1}\frac {({q}^{\lambda'_{i'}}-1)({q}^{-\lambda_{i}}-1)}{q-1}\cdot {q'}^{\alpha+i'-i}(q'-1)+
$$
$$
+\sum_{i=0}^\infty q^{\beta+1}\frac {{q}^{-\lambda_{i}}-1}{q-1}\cdot {q'}^{\alpha-i}-\sum_{i'=0}^\infty q^{\beta+1}\frac {{q}^{\lambda'_{i'}}-1}{q-1}\cdot {q'}^{\alpha+i'+1}+ qq'\frac {{q}^\beta-1}{q-1}\cdot \frac {{q'}^\alpha-1}{q'-1}
$$
\end{lem}

\begin{proof} Let us begin by noting that the character of $\BC^*\times \BC^*$ in  $\chi(\CO_{\bS}, \CO_{\bS}(-\bD_\infty+a\bD_0+b \bD_1))$ equals
\begin{equation} \label{char line bundles}
\textrm{char}_{a,b}:=qq'\cdot \frac {{q}^b-1}{q-1}\cdot \frac {{q'}^a-1}{q'-1}
\end{equation}

Recall that $J_\lambda=y^0z^{\lambda_0}+...+y^tz^{\lambda_t}$, where $t$ is large enough such that $\lambda_t=0$. Then a resolution of $J_\lambda$ is naturally given by
$$
0\rightarrow \bigoplus_{i=0}^{t-1} \CO_{\bS}(-(i+1)\bD_0-\lambda_i\bD_1) \rightarrow \bigoplus_{i=0}^t \CO_{\bS}(-i\bD_0-\lambda_i\bD_1) \rightarrow J_{\lambda}\rightarrow 0
$$
A similar resolution holds for $J_{\lambda'}$. Since both the functor $\chi$ and the character of a representation are additive in exact sequences, we have
$$
\textrm{char}_{\alpha,\beta}(\lambda',\lambda)=\sum_{i=0}^t \sum_{i'=0}^{t'} \textrm{char}_{i'-i+\alpha,\lambda'_{i'}-\lambda_i+\beta} - \sum_{i=0}^t \sum_{i'=0}^{t'-1} \textrm{char}_{i'-i+1+\alpha,\lambda'_{i'}-\lambda_i+\beta} -
$$
$$
-\sum_{i=0}^{t-1} \sum_{i'=0}^{t'} \textrm{char}_{i'-i-1+\alpha,\lambda'_{i'}-\lambda_i+\beta} + \sum_{i=0}^{t-1} \sum_{i'=0}^{t'-1} \textrm{char}_{i'-i+\alpha,\lambda'_{i'}-\lambda_i+\beta}
$$
The desired formula for $\textrm{char}_{\alpha\beta}(\lambda',\lambda)$ in the statement of the lemma is obtained by plugging \eqref{char line bundles} into the above expression, and performing some predictable computational manipulations.
\end{proof}

Recall that our fixed parabolic sheaves $\tilde{\CF'}$ and $\tilde{\CF}$ have the form \eqref{explicit tilde fixed points}. Thus one can compute the torus character in $\chi(\tilde{\CF'}, \tilde{\CF}(-\bD_\infty))$ by applying the above Lemma. Then, to obtain the torus character in the $G-$invariant part of this representation, we must keep only the terms of the form ${q'}^{nx}\cdot \dots$, and replace them by ${q'}^x \cdot \dots$. In this way, we obtain the following result.
Let $t_i$ stand for the following character of
$\widetilde{T}\times \BC^*\times \BC^*:\ (t_1,\ldots,t_i,\ldots,t_n,v,c)\mapsto
t_i$.

\begin{prop} \label{torus char E}
The character of $\widetilde{T}\times \BC^*\times \BC^*$ in $E|_{(\CF_\bullet, \CF'_\bullet)}=-\chi_G(\tilde{\CF'}, \tilde{\CF}(-\bD_\infty))$ equals
$$
\sum_{k=1}^n \sum_{l\leq k}^{l'\leq k-1} \frac{t^2_l}{t^2_{l'}}\cdot q\frac {({q}^{d'_{(k-1)l'}}-1)({q}^{-d_{kl}}-1)}{q-1}\cdot {q'}^{\lfloor \frac {-l'}n\rfloor-\lfloor \frac {-l}n\rfloor}+$$
$$+\sum_{k=1}^n \sum_{l'\leq k-1} \frac{t^2_k}{t^2_{l'}}\cdot q\frac {{q}^{d'_{(k-1)l'}}-1}{q-1}\cdot {q'}^{\lfloor \frac {-l'}n\rfloor-\lfloor \frac {-k}n\rfloor}-
$$
$$
-\sum_{k=1}^n \sum_{l\leq k}^{l'\leq k}\frac{t^2_l}{t^2_{l'}}\cdot q\frac {({q}^{d'_{kl'}}-1)({q}^{-d_{kl}}-1)}{q-1}\cdot {q'}^{\lfloor \frac {-l'}n\rfloor-\lfloor \frac {-l}n\rfloor}-$$
$$-\sum_{k=1}^n \sum_{l\leq k} \frac{t^2_l}{t^2_k}\cdot q\frac {{q}^{-d_{kl}}-1}{q-1}\cdot {q'}^{\lfloor \frac {-k}n\rfloor-\lfloor \frac {-l}n\rfloor}
$$
where $t_l:=t_{l \emph{ mod }n}$.
\end{prop}

\begin{rem} \label{rank of E}
{\em By letting $q,q',t_l\rightarrow 1$ in the above expression,
one obtains that the dimension of the fiber of $E$ equals $\sum_{i\geq j}
d_{ij}+\sum_{i\geq j} d'_{ij}=2(d_1+\dots+ d_n)=\dim(\CP_{\ul{d}})$.}
\end{rem}

\begin{rem} \label{tangent character}
{\em As was noted above, the restriction of $E$ to the diagonal is isomorphic
to the tangent bundle of $\CP_{\ul{d}}$. Therefore, the torus character in a
fixed tangent space to $\CP_{\ul{d}}$ is obtained by setting
$\widetilde{\ul{d}}=\widetilde{\ul{d}'}$ in the above expression.}
\end{rem}

In a similar way, one computes the character in the right hand side 
of~\eqref{E K-theory}.

\begin{prop} The character of $\widetilde{T}\times \BC^*\times \BC^*$ in $\chi(\CF'_{k-1-n}, \CF_{k-n}(-\bD_\infty))$ equals
$$
-\sum_{l\leq k}^{l'\leq k-1} \frac{t^2_l}{t^2_{l'}}\cdot q\frac {({q}^{d'_{(k-1)l'}}-1)({q}^{-d_{kl}}-1)}{q-1}\cdot {q'}^{\lfloor \frac {-l'}n\rfloor-\lfloor \frac {-l}n\rfloor}(q'-1)+
$$
$$
+\sum_{l'=1}^n\sum_{l\leq k} \frac{t^2_l}{t^2_{l'}}\cdot q\frac {{q}^{-d_{kl}}-1}{q-1}\cdot {q'}^{\lfloor \frac {l'-k}n\rfloor-\lfloor \frac {-l}n\rfloor}-$$ 
$$-\sum_{l=1}^n \sum_{l'\leq k-1} \frac{t^2_l}{t^2_{l'}}\cdot q\frac {{q}^{d'_{(k-1)l'}}-1}{q-1}\cdot {q'}^{\lfloor \frac {-l'}n\rfloor-\lfloor \frac {l-k-1}n\rfloor+1}
$$
Similarly, the character of $\widetilde{T}\times \BC^*\times \BC^*$ in $\chi(\CF'_{k-n}, \CF_{k-n}(-\bD_\infty))$ equals
$$
-\sum_{l\leq k}^{l'\leq k} \frac{t^2_l}{t^2_{l'}}\cdot q\frac {({q}^{d'_{kl'}}-1)({q}^{-d_{kl}}-1)}{q-1}\cdot {q'}^{\lfloor \frac {-l'}n\rfloor-\lfloor \frac {-l}n\rfloor}(q'-1)+
$$
$$
+\sum_{l'=1}^n\sum_{l\leq k} \frac{t^2_l}{t^2_{l'}}\cdot q\frac {{q}^{-d_{kl}}-1}{q-1}\cdot {q'}^{\lfloor \frac {l'-k-1}n\rfloor-\lfloor \frac {-l}n\rfloor}-$$ 
$$-\sum_{l=1}^n \sum_{l'\leq k} \frac{t^2_l}{t^2_{l'}}\cdot q\frac {{q}^{d'_{kl'}}-1}{q-1}\cdot {q'}^{\lfloor \frac {-l'}n\rfloor-\lfloor \frac {l-k-1}n\rfloor+1}
$$
\end{prop}

The above two propositions allow us to compare the left and right hand sides of \eqref{E K-theory} at the level of characters in the fixed fibers, and note that they are equal. As was mentioned at the beginning of the proof, this is enough to ensure that equality \eqref{E K-theory} holds in equivariant $K$-theory.
\end{proof}

\subsection{Relation to K\"unneth components} 
\label{end0}
The vector bundles $\chi(\CF'_k, \CF_k(-\bD_\infty))$, $\chi(\CF'_{k-1}, \CF_k(-\bD_\infty))$ can be linked to the K\"unneth components of the Chern classes 
of the universal sheaves $\ul{\CF}_k$ through a universal procedure that we illustrate below. Let $\pi:\CP_{\ul{d}}\times \bS\rightarrow \CP_{\ul{d}}$ denote the standard projection. The Beilinson spectral sequence (Theorem 3.1.3 of ~\cite{oss}) for the universal sheaf $\ul{\CF}_k$ has $E_1$ term given by
$$
E_1^{pq}=R^q\pi_*\ul{\CF}_k(\dots) \boxtimes \Omega_\bS^{-p}(\dots),
$$
where $\dots$ denote twists by line bundles on $\bS$, which do not depend on the integer $q$. This spectral sequence converges, and the $E_\infty$ term has $E_\infty^{pq}=0$ for $p+q\neq 0$, while $E_\infty^{p,-p}$ gives a filtration of $\ul{\CF}_k$ itself. As a general fact about spectral sequences, the $K$-theory class of the alternating sum $(-1)^{p+q}E_r^{pq}$ is the same for all $r$. Equating these $K$-theory classes for $r=1$ and $r=\infty$ gives us the following identity in the $K$-theory ring of $\ul{\CF}_k \times \bS$:
\begin{equation} \label{beilinson}
[\ul{\CF}_k]=\sum_p (-1)^p \left(\sum_q (-1)^q [R^q\pi_*\ul{\CF}_k(\dots)]\right) \boxtimes [\Omega_\bS^{-p}(\dots)].
\end{equation}
Consider the following complex of sheaves $\CG$ on
$\CP_{\ul{d}} \times \CP_{\ul{d}} \times \bS$:
$$
\CG:=R\CHom_{\bS}(\ul{\CF}'_k,\ul{\CF}_k(-\bD_\infty))\cong
\ul{\CF}_k(-\bD_\infty)\stackrel{L}{\otimes}_{\CO_\bS}\ul{\CF}'_k{}^\vee,
$$
where $\ul{\CF}'_k{}^\vee$ stands for the dual sheaf $R\CHom(\ul{\CF}'_k,\CO)$.
Let $\tilde{\pi}:\CP_{\ul{d}} \times \CP_{\ul{d}} \times \bS \rightarrow \CP_{\ul{d}} \times \CP_{\ul{d}}$ denote the standard projection. Since the functor $\chi$ is the derived functor of $\textrm{Hom}$, we have the obvious identity in $K$-theory:
\begin{equation} \label{chi derived 1}
[\chi(\CF'_k, \CF_k(-\bD_\infty))]=\sum_{q} (-1)^q R^q \tilde{\pi}_*[\CG].
\end{equation}
However, the class of $\CG$ is clearly given by
$$
[\CG]=[\ul{\CF}_k(-\bD_\infty)] \otimes [\ul{\CF}_k^\vee],
$$
where we suppress the notation for the obvious pull-back morphisms. We can use \eqref{beilinson} to evaluate the above:
$$
[\CG]=\sum_{p_1} \sum_{p_2} \left(\sum_q (-1)^q [R^q\pi_*\ul{\CF}_k(\dots)]
\right)\boxtimes \left(\sum_q (-1)^q [R^q\pi_*\ul{\CF}_k(\dots)^\vee]\right)
\boxtimes$$ 
$$\left([\Omega_\bS^{-p_1}(\dots)] \otimes [\Omega_\bS^{-p_2}(\dots)^\vee] \right)
$$
Plugging this into \eqref{chi derived 1} gives us
\begin{multline} \label{chi derived 2}
[\chi(\CF'_k, \CF_k(-\bD_\infty))]=\sum_{p_1} \sum_{p_2} \left(\sum_q (-1)^q [R^q\pi_*\ul{\CF}_k(\dots)]\right)\boxtimes\\
 \left(\sum_q (-1)^q [R^q\pi_*\ul{\CF}_k(\dots)^\vee]\right)\cdot Z,
\end{multline}
where $Z=\chi(\Omega_\bS^{-p_2}(\dots),\Omega_\bS^{-p_1}(\dots))$ is a constant in the $K$-theory algebra. By the Grothendieck-Riemann-Roch theorem applied to the projection map $\pi:\CP_{\ul{d}}\times \bS\rightarrow \CP_{\ul{d}}$, we obtain
$$
\textrm{ch}\left(\sum_q (-1)^qR^q\pi_*\ul{\CF}_k(\dots)\right)=\pi_*\textrm{ch}(\ul{\CF}_k\cdot \dots),
$$
where $\dots$ again denotes a class on $\bS$. Each graded piece of the
right hand side is precisely a certain polynomial in the K\"unneth components
of the tautological bundle $\ul{\CF}_k$. Thus, applying the Chern character
to~\eqref{chi derived 2}, we see that the Chern classes of the bundle $\chi(\CF'_k, \CF_k(-\bD_\infty))$ are exterior products of polynomials in the K\"unneth components of the tautological bundle $\ul{\CF}_k$. This is exactly the form we want for the expression in the right hand side of \eqref{form of diagonal}. By a similar argument, the Chern classes of $\chi(\CF'_{k-1}, \CF_k(-\bD_\infty))$ will also be exterior products of polynomials in the K\"unneth components of the 
Chern classes of the tautological bundles. \\

Now let us look at relation \eqref{E K-theory}. Since $q'=e^{\hbar'}$, this relation allows us to inductively express the Chern classes of $E$ in terms of the Chern classes of the bundles $\chi(\CF'_k, \CF_k(-\bD_\infty)), \chi(\CF'_{k-1}, \CF_k(-\bD_\infty))$. More precisely, the Chern classes of $E$ will be polynomials in the Chern classes of the $\chi$ bundles (divided by polynomials in the constant $\hbar'$, but this is allowed since we tensor all our cohomology rings by $\BC(\hbar, \hbar')$). Therefore the Chern classes of $E$ all have the form in the right hand side of \eqref{form of diagonal}. By \eqref{diagonal euler class}, $[\Delta]$ will also have that form. This completes the proof of Theorem~\ref{feiryb}.

\subsection{Torus eigenvalues in the tangent spaces} 
In this section, we will use torus character computations as in section~\ref{torus computations} in order to prove Theorem~\ref{affeigin}. In other words, we want to compute the matrix coefficients of $\fe_i, \ff_i$ in the basis $[\widetilde{\ul{d}}]$. By equivariant localization, this comes down to computing the torus character in the tangent spaces to $\sE_{\ul{d},i}$ at the torus fixed points. If we let $d'_j=d_j+\delta_{j \textrm{ mod }n}^i$, recall that the correspondence $\sE_{\ul{d},i}\subset \CP_{\ul{d}} \times \CP_{\ul{d}'}$ consists of flags $(\CF_\bullet, \CF'_\bullet)$ such that $\CF'_j=\CF_j$ for $i \neq j \textrm{ mod }n$, while $\CF'_j \subset \CF_j$ for $i = j \textrm{ mod }n$. \\

Let $\eta:\bS \rightarrow \bS$ be given by $\eta(x,y)=(x,y^{n+1})$. Let $H=\BZ/(n+1)\BZ$, which acts on $\bS=\bC \times \bX$ by multiplying $\bX$ with the roots of unity of order $n+1$. As in section~\ref{realization}, to any point $(\CF_\bullet, \CF'_\bullet)\in \sE_{\ul{d},i}$ we can associate the single $H-$invariant sheaf
$$
\CH:=\eta^*\CF_{1-n} + \eta^*\CF_{2-n}(-\bD_0)+ \dots +$$ 
$$\eta^*\CF'_{i-n}(-(i-1)\bD_0)+\eta^*\CF_{i-n}(-i\bD_0)+\dots +\eta^*\CF_{0}(-n\bD_0).
$$
This gives a realization of $\sE_{\ul{d},i}$ as the moduli space of $H-$invariant sheaves on $\bS$ satisfying a certain framing at $\bD_\infty$ and certain numerical conditions. Moreover, the tangent space to this moduli space at the point $(\CF_\bullet, \CF'_\bullet)$ is given by:
$$
T_{(\CF_\bullet, \CF'_\bullet)} \sE_{\ul{d},i}= \textrm{Ext}^1_H(\CH, \CH(-\bD_\infty)).
$$
Just like in the proof of Proposition~\ref{E bundle}, the corresponding $\textrm{Hom}$ and $\textrm{Ext}^2$ spaces vanish, and therefore we have
$$
T_{(\CF_\bullet, \CF'_\bullet)} \sE_{\ul{d},i}=-\chi_H(\CH, \CH(-\bD_\infty)).
$$
We want now to compute the character of $\widetilde{T}\times \BC^*\times \BC^*$ in this tangent space, when $(\CF_\bullet,\CF'_\bullet) \in \sE_{\ul{d},i}$ is a torus fixed point. Recall that $\CF_\bullet, \CF'_\bullet$ are given by collections of indices $\widetilde{\ul{d}},\widetilde{\ul{d'}}$, and that their components break up into direct sums as in \eqref{swim}. Then we can use Lemma~\ref{torus char 0} to compute the torus character in $\chi(\CH, \CH(-\bD_\infty))$. To obtain the $H-$invariant part of this, we keep only those terms of the form ${q'}^{(n+1)x}\cdot \dots$ and replace them with ${q'}^x\cdot \dots$. Carrying this computation through, we obtain:

\begin{prop} \label{character to correspondence}
The torus character in the tangent space to $\sE_{\ul{d},i}$ above the torus fixed point given by indices $\widetilde{\ul{d}},\widetilde{\ul{d'}}$ equals
$$
\sum_{k=1}^n \sum_{l\leq k}^{l'\leq k-1} \frac{t^2_l}{t^2_{l'}}\cdot q\frac {({q}^{d_{(k-1)l'}}-1)({q}^{-d_{kl}}-1)}{q-1}\cdot {q'}^{\lfloor \frac {-l'}n\rfloor-\lfloor \frac {-l}n\rfloor}+$$
$$+\sum_{k=1}^n \sum_{l'\leq k-1} \frac{t^2_k}{t^2_{l'}}\cdot q\frac {{q}^{d_{(k-1)l'}}-1}{q-1}\cdot {q'}^{\lfloor \frac {-l'}n\rfloor-\lfloor \frac {-k}n\rfloor}-
$$
$$
-\sum_{k=1}^n \sum_{l\leq k}^{l'\leq k}\frac{t^2_l}{t^2_{l'}}\cdot q\frac {({q}^{d_{kl'}}-1)({q}^{-d_{kl}}-1)}{q-1}\cdot {q'}^{\lfloor \frac {-l'}n\rfloor-\lfloor \frac {-l}n\rfloor}-$$
$$-\sum_{k=1}^n \sum_{l\leq k} \frac{t^2_l}{t^2_k}\cdot q\frac {{q}^{-d_{kl}}-1}{q-1}\cdot {q'}^{\lfloor \frac {-k}n\rfloor-\lfloor \frac {-l}n\rfloor}+
$$
$$
+q-{q}^{-d_{ij}+d_{(i-1)j}}+\frac{t^2_j}{t^2_i}\cdot {q}^{-d_{ij}+d_{ii}}\cdot {q'}^{\lfloor \frac {-i}n\rfloor-\lfloor \frac {-j}n\rfloor}+$$
$$+\sum_{j\neq k\leq i-1}\frac{t^2_j}{t^2_k}\cdot {q}^{-d_{ij}}\cdot ({q}^{d_{ik}}-{q}^{d_{(i-1)k}})\cdot {q'}^{\lfloor \frac {-k}n\rfloor-\lfloor \frac {-j}n\rfloor}
$$
if $d'_{ij}=d_{ij}+1$ for certain $j\leq i$.
\end{prop}

\subsection{Proof of Theorem~\ref{affeigin}}
\label{310}
Proposition~\ref{character to correspondence}
enables us to prove Theorem~\ref{affeigin}.
Indeed, by Theorem 3.7 in \cite{neg}, we have
\begin{equation}
\label{ratiu}
\fe_i [\widetilde{\ul{d}}]=-\bq_*\bp^*[\widetilde{\ul{d}}]=-\sum_{\widetilde{\ul{d}'}} [\widetilde{\ul{d}'}]\cdot \frac {\dsp \prod_{w\in T_{\widetilde{\ul{d}}}\CP_{\ul{d}}} w}{\dsp \prod_{w\in T_{(\widetilde{\ul{d}},\widetilde{\ul{d}'} )}\sE_{\ul{d},i}} w},
\end{equation}
where the notation $\prod_{w\in T_{\dots}}$ denotes the product of the torus weights in the respective tangent spaces. The sum above runs over all $\widetilde{\ul{d}'}$ such that $d'_{ij}=d_{ij}+1$ for a certain $j$. The coefficient $\fe_{i[\widetilde{\ul{d}},\widetilde{\ul{d}'}]}$ is precisely the ratio of products of weights in the above expression. To compute this ratio, note that Remark~\ref{tangent character} and Proposition~\ref{character to correspondence} imply that the character in $T_{\widetilde{\ul{d}}}\CP_{\ul{d}}$ minus the character in $T_{(\widetilde{\ul{d}},\widetilde{\ul{d}'} )}\sE_{\ul{d},i}$ is precisely
$$
-q+{q}^{-d_{ij}+d_{(i-1)j}}-\frac{t^2_j}{t^2_i}\cdot {q}^{-d_{ij}+d_{ii}}\cdot {q'}^{\lfloor \frac {-i}n\rfloor-\lfloor \frac {-j}n\rfloor}+$$
$$\sum_{j\neq k\leq i-1}\frac{t^2_j}{t^2_k}\cdot ({q}^{d_{(i-1)k}-d_{ij}}-{q}^{d_{ik}-d_{ij}})\cdot {q'}^{\lfloor \frac {-k}n\rfloor-\lfloor \frac {-j}n\rfloor}
$$
The above expression is a sum of terms of the form $\pm e^w$, where $w$ runs over the weights in the ratio we want to compute. This easily shows that the matrix coefficient of $\fe_i$ has the form claimed in Theorem~\ref{affeigin}. The matrix coefficients of $\ff_i$ are computed in the same way, and thus Theorem~\ref{affeigin} is proved.

\section{Cohomology ring of $\fM_{n,d}$}

\subsection{Giesecker moduli space} 
$\fM_{n,d}$ is the moduli space of torsion free sheaves
on $\bC\times\bX$ of rank $n$ and second Chern class $d$, trivialized at
$\bC\times\infty_\bX\cup\infty_\bC\times\bX$ (see~\cite{nak},~section~2).
We have an evident morphism
$\eta:\ \CP_{\ul{d}}\to\fM_{n,d_0}$ (forgetting the flag). It induces the
morphism on cohomology
$\eta^*:\ H^\bullet_{\widetilde{T}\times\BC^*\times\BC^*}(\fM_{n,d_0})\to
H^\bullet_{\widetilde{T}\times\BC^*\times\BC^*}(\CP_{\ul{d}})$.

\begin{lem}
\label{imaginary}
Assume ${\ul{d}}=(d,\ldots,d)$. Then
$\eta^*:\ H^\bullet_{\widetilde{T}\times\BC^*\times\BC^*}(\fM_{n,d})\to
H^\bullet_{\widetilde{T}\times\BC^*\times\BC^*}(\CP_{d,\ldots,d})$ is an
embedding.
\end{lem}

\proof Embedding the source and the target of $\eta^*$ into their localizations
and applying the Localization Theorem, it suffices to check that $\eta$ is
surjective on the sets of torus-fixed points. This is clear. \qed

\subsection{Yangian of $\fgl_n$.}
We will view $H^\bullet_{\widetilde{T}\times\BC^*\times\BC^*}(\fM_{n,d})$
as a subring of
$H^\bullet_{\widetilde{T}\times\BC^*\times\BC^*}(\CP_{d,\ldots,d})$.
Our next task is to exhibit a subalgebra $ZY(\fgl_n)\subset\fA_{\on{aff}}$
such that
$\psi(ZY(\fgl_n))=H^\bullet_{\widetilde{T}\times\BC^*\times\BC^*}(\fM_{n,d})$.

The elements $\bh_{i,r},\ i=1,\ldots,n$,
together with $\bx^\pm_{i,r},\ i=1,\ldots,n-1$, generate a copy of
$Y(\fgl_n)$ inside $\widehat{Y}$. According to \cite{mbook}, the Fourier
components of $\ba_{0,n}(u)$ generate the center $ZY(\fgl_n)$ of $Y(\fgl_n)$.
The eigenvalue of $\ba_{0,n}(u)$ on the basis vector $[\widetilde{\ul{d}}]$
is $\prod_{j\le0}(u-p_{0j}+\hbar')(u-p_{0j})^{-1}$.

\begin{thm}
\label{fM}
Assume ${\ul{d}}=(d,\ldots,d)$. Then
$\psi(ZY(\fgl_n)\otimes\BC(\hbar,\hbar')[x_1,\ldots,x_n])=
\eta^*({}^{\on{loc}}H^\bullet_{\widetilde{T}\times\BC^*\times\BC^*}
(\fM_{n,d}))$.
\end{thm}

\proof Follows from the proof of Theorem~\ref{feiryb}. The fact that the
cohomology ring of $\fM_{n,d}$ is generated by the K\"unneth components of
the Chern classes of the universal bundle $\ul{\CF}$ on
$\fM_{n,d}\times\bC\times\bX$ follows from the resolution of diagonal of
$\fM_{n,d}$ constructed in~section~4 of~\cite{nak}. \qed

\subsection{Determinant line bundle}
We consider the line bundle $\D_0$
on $\CP_{\ul{d}}$ whose fiber at the point $(\CF_\bullet)$ equals
$\det R\Gamma(\bC\times\bX,\F_0)$.

\begin{lem}\label{D_0}
$\D_0$ is a $\widetilde{T}\times\BC^*\times\BC^*$-equivariant line bundle, and
the character of $\widetilde{T}\times\BC^*\times\BC^*$ acting in the fiber of
$\D_0$ at a point $\widetilde{\ul{d}}=(d_{ij})$ equals
$\sum_{j=1}^{n}x_j(1-\sum_{k=j\mod n} d_{0k})+ \frac{1}{2}\sum_{j\le0} d_{0j}(d_{0j}-1)\hbar+\sum_{j\le0} d_{0j}\lfloor\frac{-j}{n}\rfloor\hbar'$.
\end{lem}

\begin{proof}
Straightforward.
\end{proof}

Let $\Phi_n(u)=\sum_{r\ge0}\Phi_{n,r}u^{-r-1}=\partial_u\log\ba_{0,n}(u)$.
This is the {\em noncommutative power sum of the second kind} of Gelfand et al.
(see~\cite{gel} and~\cite{mbook}~7.3 and~Corollary~1.11.8).
Let $N$ be the minimal number such that $d_{0N}\ne0$. The eigenvalue of $\Phi_n(u)$ on $[\widetilde{\ul{d}}]$ then reads \begin{multline*}\sum_{j\le0}(u-p_{0j}+\hbar')^{-1}-(u-p_{0j})^{-1}=\\=\sum_{N\le j\le0}-\hbar'u^{-2}(1-(p_{0j}-\hbar')u^{-1})^{-1}(1-p_{0j}u^{-1})^{-1}+\\ 
+\sum_{N-n\le j\le N-1} u^{-1}(1-(p_{0j}-\hbar')u^{-1})^{-1}.
\end{multline*}

In particular, we have $$\Phi_{n,1}=-\sum_{j=1}^{n} x_j-n\hbar',$$
$$\Phi_{n,2}=\sum_{j=1}^{n} (x_j+\hbar')^2-2\hbar\hbar'\sum d_{0j},$$
and
$$\Phi_{n,3}=-\sum_{j=1}^{n} (x_j+\hbar')^3 + 3\hbar\hbar'(2\sum_{j=1}^{n}x_j
\sum_{k=j\mod n} d_{0k}- \sum d_{0j}^2\hbar-\sum d_{0j}
(2\lfloor\frac{-j}{n}\rfloor-1)\hbar').$$

\begin{thm}
\label{xvi}
The operator of multiplication by the first Chern class of the
determinant line bundle $\D_0$ acts as
$$\frac{-2\Phi_{n,3}+3(\hbar-\hbar')
\Phi_{n,2}-2\sum_{j=1}^{n} (x_j+\hbar')^3-3(\hbar-\hbar')\sum_{j=1}^{n}
(x_j+\hbar')^2}{12\hbar\hbar'}+\sum_{j=1}^{n} x_j.$$
\end{thm}

\proof Due to Lemma~\ref{D_0}, the eigenvalue of $c_1(\D_0)$ on $[\widetilde{\ul{d}}]$ reads
\begin{multline*}
\sum_{j=1}^{n}x_j(1-\sum_{k=j\mod n} d_{0k})+ \frac{1}{2}\sum_{j\le0} d_{0j}(d_{0j}-1)\hbar+\sum_{j\le0} d_{0j}\lfloor\frac{-j}{n}\rfloor\hbar'=\sum_{j=1}^{n} x_j-\\
-\frac{1}{2}\left(2\sum_{j=1}^{n}x_j
\sum_{k=j\mod n} d_{0k}- \sum d_{0j}^2\hbar-\sum d_{0j}
2\lfloor\frac{-j}{n}\rfloor\hbar'+\sum d_{0j}(\hbar-\hbar')\right)=\\
=\frac{-2\Phi_{n,3}+3(\hbar-\hbar')
\Phi_{n,2}-2\sum_{j=1}^{n} (x_j+\hbar')^3-3(\hbar-\hbar')\sum_{j=1}^{n}
(x_j+\hbar')^2}{12\hbar\hbar'}+\sum_{j=1}^{n} x_j.
\end{multline*}
\qed

\end{document}